\documentclass{article}

\usepackage{booktabs, multirow, bm}
\usepackage{amsmath, amssymb, amsthm, mathrsfs}
\usepackage{indentfirst}
\usepackage{graphicx, tikz}
\usepackage{geometry}
\usepackage{caption}
\usepackage{subfigure, booktabs}

\usepackage{hyperref}

\usepackage{cleveref}
\crefname{section}{Section}{Sections}
\crefname{figure}{Figure}{Figures}
\crefname{table}{Table}{Tables}
\crefname{equation}{}{}
\crefname{theorem}{Theorem}{Theorems}
\crefname{lemma}{Lemma}{Lemmas}
\crefname{remark}{Remark}{Remarks}
\crefname{problem}{Problem}{Subproblems}

\newtheorem{theorem}{Theorem}[section]
\newtheorem{subproblem}{Subproblem}[section]

\newtheorem{lemma}{Lemma}[section]
\newtheorem{definition}{Definition}[section]
\theoremstyle{definition}
\newtheorem{example}{\noindent Example}

\graphicspath{{Figure/}}

\begin{document}
	
	\title{Co-inversion of a scattering cavity and its internal sources: uniqueness, decoupling and imaging}
	
	\author{
		Deyue Zhang\thanks{School of Mathematics, Jilin University, Changchun, China, {\it dyzhang@jlu.edu.cn}},
		Yukun Guo\thanks{School of Mathematics, Harbin Institute of Technology, Harbin, China. {\it ykguo@hit.edu.cn} (Corresponding author)},
		Yinglin Wang\thanks{School of Mathematics, Jilin University, Changchun, China, {\it yinglin19@mails.jlu.edu.cn}}
		\ and Yan Chang\thanks{School of Mathematics, Harbin Institute of Technology, Harbin, China. {\it 21B312002@stu.hit.edu.cn}}
	}
	\date{}
	\maketitle
	
\begin{abstract}
		This paper concerns the simultaneous reconstruction of a sound-soft cavity and its excitation sources from the total-field data. Using the single-layer potential representations on two measurement curves, this co-inversion problem can be decoupled into two inverse problems: an inverse cavity scattering problem and an inverse source problem. This novel decoupling technique is fast and easy to implement since it is based on a linear system of integral equations. Then the uncoupled subproblems are respectively solved by the modified optimization and sampling method. We also establish the uniqueness of this co-inversion problem and analyze the stability of our method. Several numerical examples are presented to demonstrate the feasibility and effectiveness of the proposed method.
	\end{abstract}
	
	
	\maketitle	
	
\section{Introduction}
Over the past half-century, studies on inverse scattering problems have fueled tremendous success in the interdisciplinary applications such as noninvasive detection, medical diagnostics, radar sensing, and geophysical exploration. Recently, there is a rapid surge of interest in the so-called co-inversion problems in the inverse scattering community \cite{CG22, FDTZ20, Hu, Li1, Li2, LL17, LHY2}. The goal of co-inversion is to simultaneously reconstruct multiple unknowns of distinct nature. This distinctiveness of each target unknown inherently stems from its physical feature. For example, the impinging wave is usually regarded as an active source whereas the scatterer performs a passive reaction to the excitation, hence the underlying source and scatterer should be treated as substantially different components in the scattering system. The main concern of this paper is to recover such source-scatterer pair for an interior scattering system from the internal measurements of total wave fields.

Clearly, the aforementioned co-inversion problem is closely related to the single-inversion problems of inverse scattering and inverse source problems, which have been extensively studied. As a typical example in the inverse scattering situations, the interior inverse scattering problem aims at recovering the shape of the closed cavities by interior emitters and sensors. Numerical methods to reconstruct the cavities include the linear sampling method \cite{CCM14, HuY}, the integral equation method \cite{Qin}, the factorization method \cite{LXD, MHC14}, the decomposition method \cite{Zeng} and the reciprocity gap functional method \cite{SGM16}. In particular, the co-inversion of determining  both the shape of cavity and the surface impedance is considered in \cite{QC12a}.  Theoretically, the uniqueness of the inverse cavity scattering with full data (both the intensity and phase) was studied in \cite{HuY, LXD} and a uniqueness result with phaseless data was established in \cite{Guo2} by the reference ball technique. Another active research area in the inverse scattering community is the inverse source problems for recovering various sources.  In  \cite{Bao1, BLRX15}, the authors investigated the multi-frequency inverse source problem and addressed the uniqueness as well as the stability estimates. Numerically, the Fourier methods were proposed to recover the unknown source from multi-frequency data in \cite{WGZL17, ZG15}. In \cite{BGWL20, ZGLL19}, the direct sampling methods have been developed to reconstruct the source points from the near and far field measurements, respectively. In addition, there are many other relevant works such as the source identification using multiple frequency information in \cite{Eller}, the increasing stability for inverse source scattering problem analyzed in \cite{LY}. In fact, massive more  investigations can be found in the fertile literature on the inverse scattering and inverse source problems.

Compared with the traditional single-inversion problems of determining either the source or scatterer solely, the present co-inversion problem is obviously more challenging since the unknown quantities are twofold. Hence, besides the usual difficulties of nonlinearity and ill-posedness, the co-inversion problem also suffers from a severe lack of information. Therefore, an extra data supplementation is crucial and indispensable for surmounting the data insufficiency in the underdetermined co-inversion problem. In addition, the scattered waves are forced to be repeatedly bounced back from the boundary of the impenetrable cavity, which is another obstruction of the interior problem.


In this article, we propose a decoupling-imaging scheme to recover the cavity-source pair using the total field data. To our best knowledge, this is the first attempt in the literature to simultaneously determine an impenetrable cavity and its internal sources from the measured total field data. We would like to draw the reader's attention to the following contributions in this work. First, we introduce a new and practical model setting consisting of two measurement curves. By measuring  the total data on the twinned curves, the source and scattering components can be easily decoupled by the layer potential technique. Second, the inverse source and inverse cavity subproblems can be solved separately using tailor-made approaches whilst the algorithm does not require any alternating updates between the source and cavity. Third, the inversion is easy to implement with low computational cost since there is no need for solution process of the forward problem. Finally, the applicability and effectiveness of this method is both theoretically justified and numerically tested.

The rest of this paper is organized as follows: In the next section, we introduce the mathematical formulation of the co-inversion problem and present a uniqueness result on determining both the source points and the cavity. In \cref{sec:Decoupling}, we adopt the single-layer potential theory to decouple the co-inversion problem into an inverse cavity scattering problem and an inverse source problem, and establish the stability of the decoupling. In \cref{sec:algorithms}, two imaging algorithms are developed to reconstruct the source as well as the cavity. Stability of the numerical methods is analyzed as well. Then, numerical experiments are provided in \cref{sec:numerical experiments} to verify the performance of the proposed method. Finally, some concluding remarks are given in \cref{sec:conclusion}.
	
\section{Problem setting and uniqueness}\label{sec:Problem setting}
	
Let us now introduce the mathematical model of the forward and inverse problem. In this paper we restrict ourselves to the two dimensional case and remark that the extension to the three dimensional case follows analogously. Let $D \subset\mathbb{R}^2$ be an open and simply connected domain with $C^2$ boundary $\partial D$. For a generic point $z\in D$, the incident field $u^i$ due to the point source located at $z$ is given by
\begin{equation}\label{pointsource}
		u^i (x; z)=\Phi(x,z):=\frac{\mathrm{i}}{4}H_0^{(1)}(k|x-z|), \quad x\in  D\backslash\{z\},
\end{equation}
where $H_0^{(1)}$ is the Hankel function of the first kind of order zero, and $k>0$ is the wavenumber.
	Then, the interior scattering problem for cavities can be formulated as: to find the scattered field $u^s(x;z)$ which satisfies the following boundary value problem:
	\begin{align}
		\Delta u^s+ k^2 u^s & = 0\quad \mathrm{in}\ D,\label{eq:Helmholtz} \\
		u &  = 0 \quad \mathrm{on}\ \partial D,\label{eq:boundary_condition}
	\end{align}
	where $u(x; z)=u^i(x; z)+u^s(x; z)$ denotes the total field. Assume that $k^2$ is not a Dirichlet eigenvalue for the negative Laplacian in $D$. The existence of a solution to the direct scattering problem  \eqref{eq:Helmholtz}-\eqref{eq:boundary_condition} is well known (see, e.g. \cite{Cakoni1}).
	
	In this paper, we take $B_1:=\{x\in  \mathbb{R}^2:|x|<R_1\}\subset D$ and assume that $k^2$ is not a Dirichlet eigenvalue for the negative Laplacian in $B_1$. Let $S=\cup_{j=1}^{N}\{z_j\}\subset B_1$ be a set of distinct source points with $N$ the number of source points.  Denote by $u^s(x; z_j)$ the scattered field corresponding to the incident field $u^i(x; z_j)$. Take two smooth measurement curves $\Gamma_1=\partial \Omega_1$ and $\Gamma_2=\partial \Omega_2$ such that $B_1\subset\Omega_1\subset\Omega_2\subset D$ and $k^2$ is not a Dirichlet eigenvalue for the negative Laplacian in $\Omega_2\backslash \Omega_1$. Collect the total field $u(x;z_j)=u^i(x;z_j)+u^s(x;z_j)$ on the curves $\Gamma_1$ and $\Gamma_2$. Then, the co-inversion problem under consideration is to determine the cavity-source pair $(\partial D, S)$ from the measurements $\mathbb{U}:=\{u(x; z): x\in \Gamma_1\cup\Gamma_2, z\in S\}$, namely,
	\begin{equation}\label{co-inversion}
		\mathbb{U}\rightarrow (\partial D, S).
	\end{equation}
	
	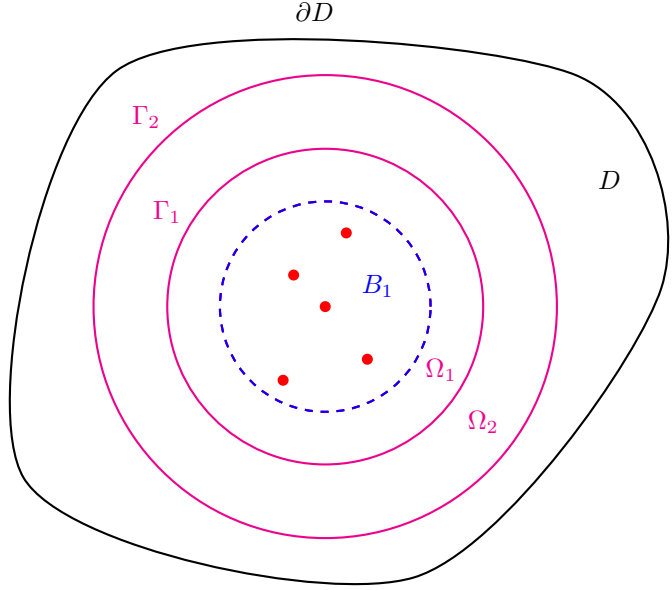
\begin{figure}
		\centering
		\begin{tikzpicture}[scale=1.4, thick]
			\pgfmathsetseed{8}
			\draw plot [smooth cycle, samples=5, domain={1:8}]
			(\x*360/8+5*rnd:2.5cm+1.5cm*rnd) node at (2.8,1.2) {$D$};
			\draw node at (0, 2.8) {$\partial D$};
			
			\draw[magenta] (0.1,0) circle (1.5cm) node at (-1.4,0.9){$ \Gamma_1 $};
			\draw[dashed] [magenta] (0.1,0) circle (1cm) node at (1.2,-0.6) {$ \Omega_1 $};
			\draw[magenta] (0.1,0) circle (2.2cm) node at (-1.6,1.8) {$ \Gamma_2 $};
			\draw[dashed] [magenta] (0.1,0) circle (1cm) node at (1.6,-1.1) {$\Omega_2 $};
			
			\filldraw [red] (-0.2,0.3) circle (1.2pt);
			\filldraw [red] (0.1,0) circle (1.2pt);
			\filldraw [red] (0.5,-0.5) circle (1.2pt);
			\filldraw [red] (0.3,0.7) circle (1.2pt);
			\filldraw [red] (-0.3,-0.7) circle (1.2pt);
			\draw[dashed] [blue] (0.1,0) circle (1cm) node at (0.6, 0.2) {$ B_1$};
			
		\end{tikzpicture}
		\caption{Illustration of the co-inversion for imaging the cavity and source points.} \label{Illustration}
	\end{figure}
	
	We refer to Figure \ref{Illustration} for an illustration of the geometry setting of the co-inversion problem\eqref{co-inversion}, and present a uniqueness result.
	
\begin{theorem}\label{Thm2.1}
		The locations of the source points $S$ can be uniquely determined by the total fields $\mathbb{U}$. Further, let $B_R$ be a ball with radius $R$ and $D\subseteq B_R$, and
		\begin{equation*}
			N_0:=\sum_{t_{0\ell}<kR} 1 +  \sum_{ t_{n\ell}<kR,  n\ne 0} 2 ,
		\end{equation*}
		where $t_{n\ell}(\ell=0,1,\cdots;n=0,1,\cdots)$ are the positive zeros of the Bessel functions $J_n$. If $N\geq N_0+1$, then the boundary of the cavity $\partial D$ can also be uniquely determined by the total fields $\mathbb{U}$.
\end{theorem}
\begin{proof}
	Assume $D_1$ and $D_2$ are two sound-soft cavities such that $(\Gamma_1\cup\Gamma_2)\subset D_\ell\subset B_R, \ell=1,2$. Let $w_1, w_2\in B_R$ and denote by $u(x; w_1)$ and $u(x; w_2)$ the total fields generated by $D_1$, $u^i(x; w_1)$ and $D_2$, $u^i(x; w_2)$, respectively. Assume that
		\begin{equation*}
			u(x; w_1)=u(x; w_2),\quad x\in \Gamma_1\cup\Gamma_2.
		\end{equation*}
		We claim that $w_1=w_2$. Otherwise, if $w_1\ne w_2 $, let $v(x)=u(x; w_1)-u(x; w_2)$, then
		\begin{align*}
			\Delta v+k^2 v & =0,\quad \mathrm{in}\ \Omega_2\backslash\Omega_1, \\
			v              & =0,\quad \mathrm{on}\  \Gamma_1\cup\Gamma_2.
		\end{align*}
		Since $ k^2 $ is not a Dirichlet eigenvalue for the negative Laplacian in $\Omega_2\backslash\Omega_1$, we have $v=0$ in $\Omega_2\backslash\Omega_1$, that is $u(x; w_1)=u(x; w_2)$ in $\Omega_2\backslash\Omega_1$. Further, the analyticity leads to
		\begin{equation*}
			u(x; w_1)=u(x; w_2),\quad \mathrm{in}\ (D_1\cap D_2)\backslash\{w_1,w_2\}.
		\end{equation*}
		By letting $x\to w_2$ and using the boundedness of $u^s(x; w_\ell) (\ell=1,2)$, we have that $u(x; w_\ell)$ is bounded and $u(x; w_2)$ tends to infinity, which is a contradiction. Hence $w_1= w_2$. This implies that the data $\{u(x; z_j): x\in \Gamma_1\cup\Gamma_2 \}$ uniquely determine the location of the source point $z_j, j=1,\cdots, N$, i.e., $S$ can be uniquely determined by $\mathbb U$.
		
		Further, if $N\ge N_0+1$, from Theorem 2.1 in \cite{Qin}, we have that  $\partial D$ can be uniquely determined by $\mathbb{U}$, which completes the proof.
	\end{proof}
	
\section{Decoupling the co-inversion problem}\label{sec:Decoupling}
	
	
	Adopting the single-layer representation, we propose a method in this section to decouple the scattered field and incident field from the measured total field. Further, the co-inversion problem can be decoupled into two subproblems: an inverse cavity scattering problem and an inverse source problem. It deserves noting that, though the incident field and the scattered field are intrinsically coupled, we can decouple them by measuring the total fields on two distinct measurement curves.
	
	\subsection{Decomposition of the total field}\label{sec:Decomposition}
	
	We assume $D\subseteq B_2:=\{x\in  \mathbb{R}^2:|x|<R_2\}$ and approximate the incident field $u^i(x; z_j)$ and the scattered field $u^s(x; z_j)$ by the following single-layer potentials, respectively,
	\begin{align}
		& u^i(x; z_j)  \approx \int_{\partial B_1}\Phi(x,y)\varphi_{j,1}(y)\mathrm{d}s(y),
		\quad x\in \mathbb R^2\backslash\overline B_1, \label{singlelayer1} \\
		& u^s(x;z_j)  \approx \int_{\partial B_2}\Phi(x,y)\varphi_{j,2}(y)\mathrm{d}s(y),
		\quad x\in B_2, \label{singlelayer2}
	\end{align}
	where $\varphi_{j,1}\in L^2(\partial B_1),\ \varphi_{j,2}\in L^2(\partial B_2)$ are unknown density functions, $j=1, \cdots, N$.
	
	To determine the density functions, we introduce the single-layer operators $\mathcal{S}_{m\ell}:L^2(\partial B_m)\to L^2(\Gamma_\ell), m,\ell=1,2$,
	$$
	(\mathcal{S}_{m\ell}\psi_m)(x):=\int_{\partial B_m}\Phi(x, y)\psi_m(y)\mathrm{d}s(y),\quad x\in \Gamma_\ell,
	$$
	where $\psi_m\in L^2(\partial B_m)$, $m=1,2$.
	
	Now, by using the total fields $u_{j, \ell}=u(x; z_j), x\in\Gamma_\ell$, $\ell=1,2$, we derive the following equations for $\varphi_{j,1}$ and $\varphi_{j,2}$,
	\begin{align*}
		\mathcal{S}_{11}\varphi_{j,1}+\mathcal{S}_{21}\varphi_{j,2} & =u_{j,1},\quad \mathrm{on}\ \Gamma_1, \\
		\mathcal{S}_{12}\varphi_{j,1}+\mathcal{S}_{22}\varphi_{j,2} & =u_{j,2},\quad\mathrm{on}\ \Gamma_2,
	\end{align*}
	or equivalently,
	\begin{equation}\label{BIE}
		\mathcal{S}\boldsymbol{\varphi}_j=\boldsymbol{u}_j,
	\end{equation}
	where $\boldsymbol{\varphi}_j=(\varphi_{j,1},\varphi_{j,2})^\top $, $\boldsymbol{u}_j=(u_{j,1}, u_{j,2})^\top$, the single-layer operator $\mathcal{S}: L^2(\partial B_1)\times L^2(\partial B_2)\to L^2(\Gamma_1)\times L^2(\Gamma_2)$ is defined by
	\begin{equation}\label{operatorS}
		\mathcal{S} :=
		\begin{bmatrix}
			\mathcal{S}_{11} &  \mathcal{S}_{21} \\
			\mathcal{S}_{12} &  \mathcal{S}_{22}
		\end{bmatrix}.
	\end{equation}

	The following theorem indicates that the operator $\mathcal{S}$ is compact and the operator equation \eqref{BIE} is ill-posed and should be solved by a regularization method.
	
	\begin{theorem} \label{Theorem3.1}
		The single-layer operator $\mathcal{S}$, defined by \eqref{operatorS}, is compact, injective and has dense range.
	\end{theorem}
	\begin{proof}
		The single-layer operator $\mathcal{S}$ is compact due to the compactness of the operators $\mathcal{S}_{ml}$, $m,l=1,2$.
		
		We next consider the injectivity of $\mathcal{S}$. Let $\mathcal{S}\boldsymbol{\psi}=\boldsymbol{0}$, where $\boldsymbol{\psi}=(\psi_1,\psi_2)^\top\in L^2(\partial B_1)\times L^2(\partial B_2)$.
		Then
		$$
		V_m(x):=\int_{\partial
			B_m}\Phi(x, y)\psi_m(y)\mathrm{d}s(y),\quad x\in\mathbb R^2\backslash \partial B_m,\  m=1,2,
		$$
		satisfy
		\begin{align*}
			\Delta(V_1+V_2)+k^2(V_1+V_2) & =0,\quad\mathrm{in}\ \Omega_2\backslash\Omega_1, \\
			V_1+V_2 & =0,\quad \mathrm{on}\  \Gamma_1\cup\Gamma_2.
		\end{align*}
		Since $k^2$ is not a Dirichlet eigenvalue for the negative Laplacian in $\Omega_2\backslash \Omega_1$, we have $V_1+V_2=0$ in $\Omega_2\backslash\Omega_1$. The analyticity of $V_m(m,\ell=1,2)$ leads to $V_1+V_2=0$ in $B_2\backslash B_1$.
		Further, we have
		\begin{align*}
			\Delta(V_1+V_2)+k^2(V_1+V_2) & =0,\quad \mathrm{in}\ B_1,\\
			V_1+V_2 & =0,\quad \mathrm{on}\  \partial B_1.
		\end{align*}
		By using the assumption that $k^2$ is not a Dirichlet eigenvalue for the negative Laplacian in $B_1$, we know $V_1+V_2=0$ in $B_1$, which, together with the jump relations of Theorem 3.1 in \cite{CK13}, yields $\psi_1=0$. Therefore, we have $V_2=0$ in $B_2$ and further the potential function $V_2$ satisfies
		\begin{align*}
			\Delta V_2+k^2V_2 & =0,\quad \mathrm{in}\ \mathbb R^2\backslash\overline B_2,\\
			V_2 & =0,\quad \mathrm{on}\  \partial B_2,
		\end{align*}
		and
		$$
		\lim_{r=|x|\to\infty}\sqrt r\left(\frac{\partial V_2}{\partial r}-\mathrm{i}kV_2\right)=0.
		$$
		And the uniqueness of solution of the exterior scattering problem implies that $V_2=0$ in $\mathbb R^2\backslash\overline B_2$. Again by using the
		jump relations of Theorem 3.1 in \cite{CK13}, we see $\psi_2=0$. Hence the operator $\mathcal{S}$ is injective.
		
		Let $ \mathcal{S}^*\boldsymbol{g}=\boldsymbol{0}$, where $ \boldsymbol{g}=(g_1,g_2)^\top\in L^2(\Gamma_1)\times L^2(\Gamma_2)$. Introduce the following potential functions	
		$$
		V_\ell^*(x):=\int_{\Gamma_l}\Phi(x, y)\overline{g_\ell(x)}\mathrm{d}s(x),\quad y\in\mathbb R^2\backslash \Gamma_\ell,\  \ell=1,2.
		$$
		Then, on one hand, it is readily to see that
		\begin{align*}
			\Delta(V_1^*+V_2^*)+k^2(V_1^*+V_2^*) & =0,\quad \mathrm{in}\ B_1,\\
			V_1^*+V_2^* & =0,\quad \mathrm{on}\  \partial B_1.
		\end{align*}
		By the assumption $k^2$ is not a Dirichlet eigenvalue for the negative Laplacian in $B_1$, we see $V_1^*+V_2^*=0$ in $B_1$. Then, from the analyticity of $V_m^*$, we know $V_1^*+V_2^*=0$ in $\Omega_1$.
		On the other hand, we see that
		\begin{align*}
			\Delta(V_1^*+V_2^*)+k^2(V_1^*+V_2^*) & =0,\quad \mathrm{in}\ \mathbb R^2\backslash\overline B_2,\\
			V_1^*+V_2^* & =0,\quad \mathrm{on}\  \partial B_2,\\
			\lim_{r=|x|\to\infty}\sqrt r\left(\frac{\partial V_2^*}{\partial r}-\mathrm{i}kV_2^*\right) & =0.
		\end{align*}
		By the uniqueness of solution of the exterior scattering problem, we have $V_1^*+V_2^*=0$ in $\mathbb R^2\backslash\overline B_2$. Then the analyticity of $V_m^*$ $(m,\ell=1,2)$ implies $V_1^*+V_2^*=0$ in $\mathbb R^2\backslash\overline \Omega_2$. Therefore, we derive that
		\begin{align*}
			\Delta(V_1^*+V_2^*)+k^2(V_1^*+V_2^*) & =0,\quad \mathrm{in}\ \Omega_2\backslash\Omega_1,\\
			V_1^*+V_2^* & =0,\quad \mathrm{on}\  \Gamma_1\cup\Gamma_2.
		\end{align*}
		Since $k^2$ is not a Dirichlet eigenvalue for the negative Laplacian in $\Omega_2\backslash\Omega_1$, we have
		$V_1^*+V_2^*=0$ in $\Omega_2\backslash\Omega_1$. Further, from the jump relations of Theorem 3.1 in \cite{CK13}, we obtain $g_1=0, g_2=0$. Hence the operator $\mathcal{S}^*$ is injective and by Theorem 4.6 in \cite{CK13} the range of $\mathcal{S}$ is dense in $L^2(\Gamma_1)\times L^2(\Gamma_2)$.
	\end{proof}
	
	Due to the ill-posedness, we need to consider the perturbed equation
	\begin{equation}\label{BIEP}
		\mathcal{S}\boldsymbol{\varphi}_j^\delta=\boldsymbol{u}_j^\delta,
	\end{equation}
	where $\boldsymbol{u}_j^\delta\in L^2(\Gamma_1)\times L^2(\Gamma_2)$ are measured noisy data satisfying $\| \boldsymbol{u}_j-\boldsymbol{u}_j^{\delta} \|_{L^2(\Gamma_1)\times L^2(\Gamma_2)} \le \delta$ with $0<\delta<1$.
	
	Seeking for a regularized solution to equation \eqref{BIEP} is to solve the following equation:
	\begin{equation}\label{Tikhonov}
		\alpha\bm{\varphi}^{\alpha,\delta}_j+\mathcal{S}^*\mathcal{S}\bm{\varphi}^{\alpha,\delta}_j=\mathcal{S}^*\bm{u}^\delta_j,
	\end{equation}
   	where the adjoint operator $\mathcal{S}^*: L^2(\Gamma_1)\times L^2(\Gamma_2)\to  L^2(\partial B_1)\times L^2(\partial B_2)$ is defined by
   \begin{equation*}
   	\mathcal{S^*} :=
   	\begin{bmatrix}
   		\mathcal{S}^*_{11} &  \mathcal{S}^*_{12} \\
   		\mathcal{S}^*_{21} &  \mathcal{S}^*_{22}
   	\end{bmatrix},
   \end{equation*}
   with $\mathcal{S}^*_{m\ell}: L^2(\Gamma_\ell)\to L^2(\partial B_m), m,\ell=1,2$, given by
    $$
    (\mathcal{S}^*_{m\ell}g_\ell)(y):=\int_{\Gamma_\ell}\overline{\Phi(x, y)}g_\ell(x)\mathrm{d}s(x),\quad
    y\in \partial B_m,\quad g_\ell\in L^2(\Gamma_\ell), \ell=1,2.
    $$

	The regularized solution to \eqref{BIEP} is the unique minimum of the Tikhonov functional
	\begin{equation*}
		J_{\alpha}(\boldsymbol{\varphi}_j):=\left\| \mathcal{S}\boldsymbol{\varphi}_j-\bm{u}^\delta_j \right\|^2_{L^2(\Gamma_1)\times L^2(\Gamma_2)}+\alpha\left\| \boldsymbol{\varphi}_j\right \|^2_{L^2(\partial B_1)\times L^2(\partial B_2)}.
	\end{equation*}
	In this paper, the regularization parameter $\alpha=\alpha(\delta)>0$ is chosen by the Morozov's discrepancy principle, and we obtain the following regularized approximation on the incident field and the scattered field,
	\begin{align}
		& u^i_{\alpha(\delta),\delta}(x; z_j)  = \int_{\partial B_1}\Phi(x,y)\varphi^{\alpha(\delta),\delta}_{j,1}(y)\mathrm{d}s(y),
		\quad x\in \mathbb R^2\backslash\overline B_1, \label{Tikhonov_ui}\\
		& u^s_{\alpha(\delta),\delta}(x; z_j)  = \int_{\partial B_2}\Phi(x, y)\varphi^{\alpha(\delta),\delta}_{j, 2}(y)\mathrm{d}s(y),
		\quad x\in B_2. \label{Tikhonov_us}
	\end{align}
	
	\subsection{Stability of the decomposition}
	In this subsection, we will give the error estimates of the decomposition. To this aim, we first introduce some
	single-layer operators $\mathcal{S}_m:L^2(\partial B_m)\to L^2(\Sigma_m)$, $ m=1,2$,
	$$
	(\mathcal{S}_m\psi_m)(x) =\int_{\partial B_m}\Phi(x,y)\psi_m(y)\mathrm{d}s(y),\quad x\in \Sigma_m,
	$$
	where $\Sigma_m=\partial \Omega_{\Sigma_m}$ and $B_1\subseteq \Omega_{\Sigma_m} \subseteq B_2$.
	
	Since $k^2$ is not a Dirichlet eigenvalue for the negative Laplacian in $B_1$, this following lemma is a direct result of Theorem 5.20 in \cite{CK13}.
	\begin{lemma}\label{lemma1}
		The single-layer operator $\mathcal{S}_1$ is injective and has dense range.
	\end{lemma}

   Following the proof of  Theorem \ref{Theorem3.1}, we readily derive the following result and the proof is omitted.
	\begin{lemma}\label{lemma2}
		The single-layer operator $\mathcal{S}_2$ is injective and has dense range provided that $k^2$ is not a Dirichlet eigenvalue for the negative Laplacian in $\Omega_{\Sigma_2}$.
	\end{lemma}
	
	By using Lemma \ref{lemma1} and Lemma \ref{lemma2}, it is readily to see that for a sufficiently small positive constant $\varepsilon_0$  $(0<\varepsilon_0\ll 1)$, there exist density functions $\varphi_{j,1}^*\in L^2(\partial B_1)$ and $\varphi_{j,2}^*\in L^2(\partial B_2)$, such that
	\begin{align*}
		\left\|\mathcal{S}_1\varphi_{j,1}^*-u^i(\ \cdot\ ;z_j)\right\|_{L^2(\Gamma_1)} & <\varepsilon_0,\\
		\left\|\mathcal{S}_2\varphi_{j,2}^*-u^s(\ \cdot\ ;z_j)\right\|_{L^2(\partial D)} & <\varepsilon_0 .
	\end{align*}
	Denote $\bm{\varphi}_j^*=(\varphi_{j,1}^*,\varphi_{j,2}^*)^\top$, and from continuous dependence of solutions on the boundary value, we obtain
	\begin{equation}\label{BIEAppro}
		\mathcal{S}\bm{\varphi}_j^*=\bm{u}_j^*,
	\end{equation}
	where
	\begin{equation}\label{varepsilon}
		\left	\|\boldsymbol{u}_j^*-\boldsymbol{u}_j \right\|_{L^2(\Gamma_1)\times L^2(\Gamma_2)}<C_0\varepsilon_0=:\varepsilon,
	\end{equation}
    with $C_0>0$.
	
	In the following, we present the main result on the error estimates.
	
	\begin{theorem}\label{Theorem3.2}
		Let $\delta+\varepsilon \le \left\| \boldsymbol{u}_j^{\delta} \right\|_{L^2(\Gamma_1)\times L^2(\Gamma_2)}$ with $\varepsilon$ being defined in \eqref{varepsilon}, the regularized solution $\boldsymbol{\varphi}^{\alpha(\delta),\delta}_j $ of \eqref{Tikhonov} satisfy $\left\|\mathcal{S}\boldsymbol{\varphi}^{\alpha(\delta),\delta}_j-\boldsymbol{u}_j^{\delta}\right\|_{L^2(\Gamma_1)\times L^2(\Gamma_2)}=\delta+\varepsilon$, $\delta\in(0,\delta_0)$. Then
		
		\noindent(a) There exists a function $\boldsymbol{h}_j\in L^2(\Gamma_1)\times L^2(\Gamma_2)$ such that
		$$
		\left\| \boldsymbol{\varphi}^*_j-\mathcal{S}^*\boldsymbol{h}_j\right \|_{L^2(\partial B_1)\times L^2(\partial B_2)}<\varepsilon ;
		$$
		\noindent(b) Let $ \| \boldsymbol{h}_j \|_{L^2(\Gamma_1)\times L^2(\Gamma_2)} \le E_j $, then
		$$
		\left\| \boldsymbol{\varphi}^{\alpha(\delta),\delta}_j-\boldsymbol{\varphi}_j^* \right\|_{L^2(\partial B_1)\times L^2(\partial B_2)} \le 2\varepsilon+2\sqrt{(\delta+\varepsilon)E_j} .
		$$
	\end{theorem}
	\begin{proof}
		\noindent(a) The range of $ \mathcal{S}^*$ is dense in $L^2(\partial B_1)\times L^2(\partial B_2)$, since $\mathcal{S}: L^2(\partial B_1)\times L^2(\partial B_2)\to L^2(\Gamma_1)\times L^2(\Gamma_2)$ is compact and injective by Theorem \ref{Theorem3.1}. Therefore, for $\varepsilon$, there exists a function $ \boldsymbol{h}_j \in L^2(\Gamma_1)\times L^2(\Gamma_2)$ such that $\|\boldsymbol{\varphi}_j^*-\mathcal{S}^*\boldsymbol{h}_j \|_{L^2(\partial B_1)\times L^2(\partial B_2)}< \varepsilon$.
		
		\noindent(b) Let $\boldsymbol{\varphi}_j^{\delta}:= \boldsymbol{\varphi}^{\alpha(\delta),\delta}_j$ be the minimum of the Tikhonov functional
		\begin{equation*}
			J^{\delta}(\boldsymbol{\varphi}_j):=J_{\alpha(\delta),\delta}(\boldsymbol{\varphi}_j)=\left\| \mathcal{S}\boldsymbol{\varphi}_j-\boldsymbol{u}_j^{\delta}\right \|^2_{L^2(\Gamma_1)\times L^2(\Gamma_2)}
			+\alpha(\delta)\left\| \boldsymbol{\varphi}_j \right\|^2_{L^2(\partial B_1)\times L^2(\partial B_2)}.
		\end{equation*}
		Then, we have
		\begin{align*}
			\quad (\delta+\varepsilon)^2+\alpha(\delta)\left\| \boldsymbol{\varphi}_j^{\delta}\right \|^2_{L^2(\partial B_1)\times L^2(\partial B_2)}
			& =J^{\delta}(\boldsymbol{\varphi}_j^{\delta}) \le  J^{\delta}(\boldsymbol{\varphi}_j^*) \\
			& =\left\| \boldsymbol{u}_j^*-\boldsymbol{u}_j^{\delta} \right\|^2_{L^2(\Gamma_1)\times L^2(\Gamma_2)}
			+\alpha(\delta)\left\| \boldsymbol{\varphi}_j^* \right\|^2_{L^2(\partial B_1)\times L^2(\partial B_2)} \\
			&\le (\delta+\varepsilon)^2+\alpha(\delta)\left\| \boldsymbol{\varphi}_j^* \right\|^2_{L^2(\partial B_1)\times L^2(\partial B_2)},
		\end{align*}
		which implies $\left\|\boldsymbol{\varphi}_j^{\delta} \right\|_{L^2(\partial B_1)\times L^2(\partial B_2)} \le \left\| \boldsymbol{\varphi}_j^* \right\|_{L^2(\partial B_1)\times L^2(\partial B_2)}$ for all $\delta >0$. Hence
		\begin{align*}
			\quad \left\| \boldsymbol{\varphi}_j^{\delta}-\boldsymbol{\varphi}_j^* \right\|^2_{L^2(\partial B_1)\times L^2(\partial B_2)}
			&=\left\| \boldsymbol{\varphi}_j^{\delta} \right\|^2_{L^2(\partial B_1)\times L^2(\partial B_2)}-2\Re\langle\boldsymbol{\varphi}_j^{\delta}, \boldsymbol{\varphi}_j^*\rangle
			+\left\| \boldsymbol{\varphi}_j^* \right\|^2_{L^2(\partial B_1)\times L^2(\partial B_2)} \\
			&\le 2 \left( \left\| \boldsymbol{\varphi}_j^*\right \|^2_{L^2(\partial B_1)\times L^2(\partial B_2)}-\Re\langle\boldsymbol{\varphi}_j^{\delta}, \boldsymbol{\varphi}_j^*\rangle \right)\\
			&=2\Re\langle\boldsymbol{\varphi}_j^*-\boldsymbol{\varphi}_j^{\delta}, \boldsymbol{\varphi}_j^*\rangle .
		\end{align*}
	    where $\langle \cdot , \cdot \rangle$ denotes the $L^2$-inner product on $\partial B_1\times \partial B_2$.
	
		From (a), let $ \tilde{\bm{\varphi}}_j =\mathcal{S}^*\boldsymbol{h}_j\in L^2(\partial B_1)\times L^2(\partial B_2)$ such that $\left\|\tilde{\boldsymbol{\varphi}}_j-\boldsymbol{\varphi}^*\right\|_{L^2(\partial B_1)\times L^2(\partial B_2)}\le \varepsilon $. Since
		$
		\left\|\boldsymbol{u}_j^*-\boldsymbol{u}_j^\delta \right\|_{L^2(\Gamma_1)\times L^2(\Gamma_2)} \leq \left\|\boldsymbol{u}_j^*-\boldsymbol{u}_j \right\|_{L^2(\Gamma_1)\times L^2(\Gamma_2)}+\left\|\boldsymbol{u}_j-\boldsymbol{u}_j^\delta \right\|_{L^2(\Gamma_1)\times L^2(\Gamma_2)}<\varepsilon+\delta
		$, we obtain
		\begin{align*}
			\left \| \boldsymbol{\varphi}_j^{\delta}-\boldsymbol{\varphi}_j^* \right\|^2_{L^2(\partial B_1)\times L^2(\partial B_2)}
			&\le 2\Re\langle\boldsymbol{\varphi}_j^*-\boldsymbol{\varphi}_j^{\delta}, \boldsymbol{\varphi}_j^*-\tilde{\boldsymbol{\varphi}}_j\rangle+2\Re\langle\boldsymbol{\varphi}_j^*-\boldsymbol{\varphi}_j^{\delta}, \mathcal{S}^*\boldsymbol{h}_j\rangle \\
			&\le 2\varepsilon \left\| \boldsymbol{\varphi}_j^*-\boldsymbol{\varphi}_j^{\delta} \right\|_{L^2(\partial B_1)\times L^2(\partial B_2)}  +2\Re\langle\boldsymbol{u}_j^*-\mathcal{S}\boldsymbol{\varphi}_j^{\delta}, \boldsymbol{h}_j\rangle  \\
			&\le 2\varepsilon\left\| \boldsymbol{\varphi}_j^*-\boldsymbol{\varphi}_j^{\delta} \right\|_{L^2(\partial B_1)\times L^2(\partial B_2)} + 2\Re\langle\boldsymbol{u}_j^*-\boldsymbol{u}_j^{\delta}, \boldsymbol{h}_j\rangle  \\
			&\quad +2\Re\langle\boldsymbol{u}_j^{\delta}-\mathcal{S}\boldsymbol{\varphi}_j^{\delta}, \boldsymbol{h}_j\rangle \\
			&\le 2\varepsilon\left\| \boldsymbol{\varphi}_j^*-\boldsymbol{\varphi}_j^{\delta} \right\|_{L^2(\partial B_1)\times L^2(\partial B_2)}+4(\delta+\varepsilon)E_j.
		\end{align*}
		This means $ \left( \left\| \boldsymbol{\varphi}_j^*-\boldsymbol{\varphi}_j^{\delta} \right\|_{L^2(\partial B_1)\times L^2(\partial B_2)}-\varepsilon \right)^2 \le \varepsilon^2+4(\delta+\varepsilon)E_j $, and thus,
		$$
		\left\| \boldsymbol{\varphi}_j^{\delta}-\boldsymbol{\varphi}_j^* \right\|_{L^2(\partial B_1)\times L^2(\partial B_2)} \le 2\varepsilon+2\sqrt{(\delta+\varepsilon)E_j} .
		$$
		This completes the proof.
	\end{proof}
	
	From Theorem \ref{Theorem3.2} and $\varepsilon\ll 1$, we obtain the error estimates of the decomposition:
	\begin{align}
		\max_{1\leq j\leq N}\left\|u^i_{\alpha(\delta),\delta}(\cdot;z_j)-u^i(\cdot;z_j)\right\|_{L^2(\partial \Omega_\Sigma)} & \leq   C_1\sqrt{\delta+\varepsilon}, \label{errorui} \\
		\max_{1\leq j\leq N}\left\|u^s_{\alpha(\delta),\delta}(\cdot;z_j) -u^s(\cdot;z_j)\right\|_{L^2(\partial D)} & \leq   C_2\sqrt{\delta+\varepsilon}, 	\label{errorus}
	\end{align}
	where $\Omega_\Sigma\supset \Omega_1$, $C_1=C_1(\partial \Omega_\Sigma)$ and $C_2=C_2(\partial D)$ are positive constants.

	\subsection{Uncoupled subproblems}
	
	Once the incident and scattering components have been decoded from the measurements of the total field, the co-inversion problem \eqref{co-inversion} can be completely decoupled into the following two inverse problems:
	
	\begin{subproblem}\label{subPro1}
		Reconstruct the boundary of the cavity $\partial D$ from the decoupled data in \eqref{Tikhonov_ui}-\eqref{Tikhonov_us}, $\left\{u^i_{\alpha(\delta),\delta}(x;z_j)\right\}_{j=1}^N$ and $\left\{u^s_{\alpha(\delta),\delta}(x;z_j)\right\}_{j=1}^N$.
	\end{subproblem}
	
	\begin{subproblem}\label{subPro2}
		Determine the locations of the source points $S$ from the decoupled data in \eqref{Tikhonov_ui}, $\left\{u^i_{\alpha(\delta),\delta}(x;z_j)\right\}_{j=1}^N$.
	\end{subproblem}

     Numerical methods for solving these subsequent problems will be presented in the next section.
	
	\section{Imaging algorithms}\label{sec:algorithms}
	The aim of this section is to develop imaging algorithms to separately solve \cref{subPro1} and \cref{subPro2}, namely, find the shape of cavity and determine the source locations. Stability of the numerical methods will be analyzed as well.
	Motivated by the classical optimization method for the inverse obstacle scattering problem \cite{CK13}, in subsection \ref{subsec:algorithm1} we propose an optimization method for solving \cref{subPro1}. Then a direct sampling method is developed to solve \cref{subPro2} in subsection \ref{subsec:algorithm2}. We would like to point out that, in the following optimization (resp. sampling) method, the incident field in the cost function (resp. the imaging function) is an intermediate quantity decoupled numerically from the total field. This is essentially different from the conventional optimization or sampling method, where the incident field is known in advance.
	
	\subsection{Optimization for imaging the cavity}\label{subsec:algorithm1}
	In this subsection, we present the optimization method for reconstructing the boundary of the cavity. The boundary $\partial D$ is then sought for by minimizing the defect
	\begin{equation}\label{Optimization}
		\sum_{j=1}^{N}\left\|u^i_{\alpha(\delta),\delta}(\ \cdot\ ;z_j) + u^s_{\alpha(\delta),\delta}(\ \cdot\ ;z_j)\right\|_{L^2(\Gamma)}
	\end{equation}
	over some class $U$ of admissible curves $\Gamma$. Here the admissible class $U$ is the compact set (with respect to the $C^{1,\beta}$ norm, $0 < \beta < 1$) of all star-like closed $C^2$ curves, described by
	\begin{equation*}
		\Gamma=\{r(\hat{x})\hat{x}: \hat{x}\in \mathbb{S}\}, \ r\in C^2(\mathbb{S})
	\end{equation*}
	with the a prior information
	\begin{equation}\label{eq:priori}
		0<a\leq r(\hat{x})\leq b, \ \hat{x}\in \mathbb{S},
	\end{equation}
	where $a$ and $b$ are positive constants, $\mathbb{S}=\{x\in\mathbb{R}^2:|x|=1\}.$
	
	To investigate the convergence properties of the optimization method, we define the cost functional $\mu: U \to \mathbb R$ by
	\begin{equation}\label{costfunctional}
		\mu(\Gamma;\delta)=\sum_{j=1}^{N}
		\left\|\mathcal{S}_1\varphi_{j,1}^{\alpha(\delta),\delta}+\mathcal{S}_2\varphi_{j,2}^{\alpha(\delta),\delta}\right \|^2_{L^2(\Gamma)} ,
	\end{equation}
	where $\Gamma\in U$, and introduce the following definition of optimal curve.
	\begin{definition}\label{Opmcurve}
		Given the measured total field $\boldsymbol{u}_j^{\delta}, j=1,\cdots, N$ and a regularization parameter $\alpha(\delta)>0$, a curve $\Gamma^*\in U$ is called optimal if $\Gamma^*$ minimizes the cost functional \eqref{costfunctional}, i.e.,
		$$
		\mu(\Gamma^*; \delta):=m(\delta),
		$$
		where
		$$
		m(\delta)={\mathop{\inf}\limits_{\Gamma\in U}}\ \mu(\Gamma;\delta).
		$$
	\end{definition}
	
	In terms of Definition \ref{Opmcurve}, the following theorem holds.
	\begin{theorem}
		Assume $\partial D$ belongs to $U$, then for each $\delta>0$, there exists an optimal curve $\Gamma^*\in U$, satisfying
		\begin{equation*}
			\mu(\Gamma^*;\delta)\le C\sqrt{\delta+\varepsilon},
		\end{equation*}
		where $C$ is a positive constant.
	\end{theorem}
	\begin{proof}
		Let $\{\Gamma^n\}$ be a minimizing sequence in $U$, i.e.,
		$$
		\lim_{n\to\infty}\mu(\Gamma^n;\delta)=m(\delta).
		$$
		Since $U$ is compact, we can assume that $\Gamma^n\to\Gamma^*$, $n\to\infty$. Further, by $\partial D\in U$, we have
		\begin{equation*}
			\mu(\Gamma^*;\delta)\le \mu(\partial D;\delta)=\sum_{j=1}^{N}\left\|\mathcal{S}_1\varphi_{j,1}^{\alpha(\delta),\delta}+\mathcal{S}_2\varphi_{j,2}^{\alpha(\delta),\delta} \right\|^2_{L^2(\partial D)}.
		\end{equation*}
		From \eqref{varepsilon}, \eqref{errorui} and \eqref{errorus}, we deduce
		\begin{align*}
			\left\|\mathcal{S}_1\varphi_{j,1}^{\alpha(\delta),\delta}+\mathcal{S}_2\varphi_{j,2}^{\alpha(\delta),\delta} \right\|^2_{L^2(\partial D)} &\le \left\|\mathcal{S}_1\varphi_{j,1}^{\alpha(\delta),\delta}-u^i(\ \cdot\ ;z_j) \right\|^2_{L^2(\partial D)}+\left\|\mathcal{S}_2\varphi_{j,2}^{\alpha(\delta),\delta}-u^s(\ \cdot\ ;z_j) \right\|^2_{L^2(\partial D)}\\
			&\le \left\|\mathcal{S}_1\varphi_{j,1}^{\alpha(\delta),\delta}-\mathcal{S}_1\varphi_{j,1}^*\right\|^2_{L^2(\partial D)}+\left\|\mathcal{S}_1\varphi_{j,1}^*-u^i(\ \cdot \ ;z_j) \right\|^2_{L^2(\partial D)} \\
			&\quad+\left\|\mathcal{S}_2\varphi_{j,2}^{\alpha(\delta),\delta}-\mathcal{S}_2\varphi_{j,2}^*\right\|^2_{L^2(\partial D)}+\left\|\mathcal{S}_2\varphi_{j,2}^*-u^s(\ \cdot \ ; z_j) \right\|^2_{L^2(\partial D)}
			\\
			&\le 2C_3(\sqrt{\delta+\varepsilon}+\varepsilon),
		\end{align*}
		where $C_3>0$. This implies
		$$
		\mu(\Gamma^*; \delta)\le C\sqrt{\delta+\varepsilon}.
		$$
	\end{proof}
	
	\subsection{Direct sampling for locating the point sources}\label{subsec:algorithm2}
	
	In this subsection, a direct sampling method is proposed for determining the locations of the source points.
	
	For any sampling point $y\in B_1$, we introduce the following indicator functions
	\begin{equation}\label{indicator}
		I_j(y)=\Re \left\{ \left\langle u_{\alpha(\delta),\delta}^i(\ \cdot \ ;z_j),\mathrm{e}^{\mathrm{i}(k|\cdot-y|+\frac{\pi}{4})}\right\rangle_{L^2(\Gamma_3)} \right\}, \quad  j=1,\cdots,N,
	\end{equation}
	where $\left\langle \cdot , \cdot \right\rangle$ denotes the $L^2$-inner product, and $\Gamma_3:=\{x\in  \mathbb{R}^2:|x|=R_3\}$ such that $R_3\gg R_1$. We take the maximum point $z_j^*\in B_1$ of the indicator function $I_j(y)$ for each $j$ as an approximation of the exact source point $z_j$.
	
	The following theorem sheds light on the indicating properties of $I_j(y),  j=1, \cdots, N$.
	
	\begin{theorem}\label{ThmIndicator}	
		For all $y\in B_1$, we have
		\begin{align*}
			I_j(y)&= \frac{1}{2\sqrt{2\pi k}}\int_{\Gamma_3}\frac{\cos(k(|x-z_j|-|x-y|))}{\sqrt{|x-z_j|}}\mathrm{d}s(x) \\
			&\quad +\mathcal{O}\left(k^{-1}(R_3-R_1)^{-3/2}\right)+\mathcal{O}\left((\delta+\varepsilon)^{1/2}\right),\quad j=1,\cdots,N.
		\end{align*}
	\end{theorem}
	\begin{proof}
		From (3.82) in \cite{CK13}, we have the following asymptotic behavior of the Hankel functions
		$$
		H_0^{(1)}(t)=\sqrt{\frac{2}{\pi t}}\mathrm{e}^{\mathrm{i}(t-\frac{\pi}{4})}
		\left\{1+\mathcal{O}\left(\frac{1}{t}\right)\right\},\quad t\to\infty.
		$$
		Then, we deduce
		\begin{align*}
			&\quad\Re \left\{ \left\langle u^i(x; z_j),\mathrm{e}^{\mathrm{i}(k|x-y|+\frac{\pi}{4})}\right\rangle_{L^2(\Gamma_3)} \right\} \\
			&=\Re\left\{\int_{\Gamma_3}u^i(x;z_j)\ \mathrm{e}^{-\mathrm{i}(k|x-y|+\frac{\pi}{4})}\mathrm{d}s(x) \right\} \\
			&=\Re\left\{\int_{\Gamma_3}\frac{\mathrm{i}}{4}H_0^{(1)}(k|x-z_j|)\ \mathrm{e}^{-\mathrm{i}(k|x-y|-\frac{\pi}{4})}\mathrm{e}^{-\mathrm{i}\frac{\pi}{2}}\mathrm{d}s(x) \right\} \\
			&=\frac{1}{4}\Re\left\{\int_{\Gamma_3}\sqrt{\frac{2}{\pi k |x-z_j|}}\ \mathrm{e}^{\mathrm{i}k(|x-z_j|-|x-y|)}\mathrm{d}s(x) \right\}+\mathcal{O}\left(k^{-1}(R_3-R_1)^{-3/2}\right) \\
			&=\frac{\sqrt{2}}{4\sqrt{\pi k}}\int_{\Gamma_3}\frac{\cos(k(|x-z_j|-|x-y|))}{\sqrt{|x-z_j|}}\mathrm{d}s(x)+\mathcal{O}\left(k^{-1}(R_3-R_1)^{-3/2}\right).
		\end{align*}
		Further, by \eqref{varepsilon}, \eqref{errorui} and the Schwarz inequality, we have for each $y\in B_1$,
		\begin{align*}
			I_j(y)
			&= \Re\left\{\int_{\Gamma_3}\left(u_{\alpha(\delta),\delta}^i(x; z_j)-u^i(x; z_j)\right)
			\mathrm{e}^{-\mathrm{i}\left(k|x-y|+\frac{\pi}{4}\right)}\mathrm{d}s(x) \right.\\
			&\left.  \quad+\int_{\Gamma_3}u^i(x; z_j)\mathrm{e}^{-\mathrm{i}\left(k|x-y|+\frac{\pi}{4}\right)}\mathrm{d}s(x) \right\} \\
			&=\mathcal{O}\left((\delta+\varepsilon)^{1/2}\right)+\frac{\sqrt{2}}{4\sqrt{\pi k}}\int_{\Gamma_3}\frac{\cos(k(|x-z_j|-|x-y|))}{\sqrt{|x-z_j|}}\mathrm{d}s(x)\\
			& +\mathcal{O}\left(k^{-1}(R_3-R_1)^{-3/2}\right).
		\end{align*}
		This completes the proof.
	\end{proof}
	
	In virtue of the above theorem, we see that each function $I_j(y)$ should decay as the sampling point $y$ recedes from the corresponding source point $z_j$. In particular, for small $\delta$, function $I_j(y)$ takes the local maximum value at $y=z_j$,
	$$
	I_j(z_j)\approx \frac{\sqrt{2}}{4\sqrt{\pi k}}\int_{\Gamma_3}\frac{1}{\sqrt{|x-z_j|}}\mathrm{d}s(x), \quad j=1,\cdots,N.
	$$
	This indicating behavior will be numerically tested by the experiments in the next section. 
	
	We end this section with a brief description of our algorithm for the co-inversion problem.
	
	\begin{table}
		\centering
		\begin{tabular}{cp{.8\textwidth}}
			\toprule
			\multicolumn{2}{l}{{\bf Algorithm :} Imaging the cavity and its internal point sources.} \\
			\midrule
			{\bf Step 1} &  {\bf Data bi-collection}: take two non-intersecting and closed curves $\Gamma_1$ and $\Gamma_2$ which lie inside $D$ and encompass the source points; Record the noisy total field data  $\bm{u}_j^{\delta}|_{\Gamma_1\cup\Gamma_2 }$,  $j=1,\cdots,N$; \\
			{\bf Step 2} &  {\bf Decoupling}: For $j=1,\cdots, N$, solve the regularized equation \eqref{Tikhonov} with the Morozov's discrepancy principle to find the density $\bm{\varphi}^{\alpha(\delta),\delta}_j$; Then the approximate incident field and scattered field are represented by $u^i_{\alpha(\delta),\delta}$ and $u^s_{\alpha(\delta),\delta}$ via \eqref{Tikhonov_ui} and\eqref{Tikhonov_us}, respectively.	\\
			
			{\bf Step 3} & {\bf Reconstruction.} The uncoupled inversions can be implemented separately: \\
			&	{\bf A: optimization for cavity.}  Select an initial curve $\Gamma^{(0)}\subseteq(B_2\backslash B_1)$ for the boundary $\partial D$, and reconstruct the boundary $\partial D$ by minimizing \eqref{Optimization} over the admissible curves;\\
			& {\bf B: sampling for source.} Choose $\Gamma_3$ and generate a suitable sampling grid $\mathcal{T}$; For $j=1,\cdots,N$,  evaluate the indicator $I_j(y)$ defined in \eqref{indicator} over $\mathcal{T}$  and take the maximum point $z_j^*$ as the reconstruction of $z_j$. \\
			\bottomrule
		\end{tabular}
	\end{table}	

	\section{Numerical verifications}\label{sec:numerical experiments}
	In this section, we shall conduct a variety of numerical experiments to verify the applicability and effectiveness of our method.
	
	In order to obtain the synthetic total field, we adopt the boundary integral equation method to numerically solve the direct problem \eqref{eq:Helmholtz}-\eqref{eq:boundary_condition} to generate the scattered field and the Nystr\"{o}m method \cite{CK13} is utilized to solve the boundary integral equation. By adding the incident field to the scattered field, we obtain the simulated total field data. For the sake of testing the stability of proposed methods, we further corrupt the total field data ${\bm u}_j$ by adding random noise such that
	$$
	{\bm u}_j^{\delta}:={\bm u}_j+\gamma_1\delta|{\bm u}_j|\mathrm{e}^{\mathrm{i}\pi\gamma_2} ,
	$$
	where $\gamma_1,\gamma_2$ are two uniformly distributed random numbers ranging from $-1$ to $1$ and  $0<\delta<1$ is the noise level.
	
	In our simulations, the parameter $\varepsilon$ in \eqref{varepsilon} is set to be $10^{-16},$ which can be considered to be negligible compared with the discretization error. The integrals over the two auxiliary curves $\partial B_1$ and $\partial B_2$ are numerically approximated by the trapezoidal rule with 90 and 160 grid points, respectively. The regularization parameter $\alpha$ in equation \eqref{Tikhonov} is automatically determined by the Morozov's discrepancy principle.
	
In the direct sampling method for determining the source points, $R_3=15$ is used and the sampling domain is chosen to be $[-1,1]\times[-1,1]$ with $150\times150$ equally distributed sampling points. To reconstruct the boundary $\partial D$ by the optimization method, we solve the nonlinear least-squares problem by the Levenberg-Marquardt algorithm with functional value stopping criterion and successive iterate stopping criterion chosen to be $10^{-6}$ and $10^{-5},$ respectively. The initial curve for the optimization method is chosen to be the unit circle centered at the origin. To formulate the admissible curves in the optimization method, we approximate the cavity by star-like domains whose radial functions $\tilde{r}(t)$ are represented by trigonometric polynomials of degree less than or equal to 8, i.e.,
	\[
	\tilde{r}(t)=a_0+\sum_{j=1}^8(a_j\cos jt+b_j\sin jt).
	\]
	
	In order to quantify the accuracy of reconstructions, we utilize a discrete version of the $L^2$ relative error
	$$
		\frac{\|r^*-\tilde{r}\|_{L^2([0,2\pi])}}{\|r^*\|_{L^2([0,2\pi])}},
	$$
	where $r^*$ and $\tilde{r}$ signify the exact and reconstructed boundary curves, respectively.
	
In the subsequent figures with regard to the geometry setting of the problem, the black solid curves denote the boundaries of the exact cavities,  and the red points indicate the exact source points. The two blue dashed circles stand for the interior measured lines. The black dashed circles located inside and outside of the boundaries, respectively, designate the auxiliary curves. In the subfigures illustrating the reconstructions, the exact and the reconstructed boundaries are displayed as the black solid lines and the red dashed lines, respectively. The exact and reconstructed source points are designated with the black `+' marks and red small circles, respectively.
	
\begin{example}\label{example1}
This example is designed to verify the performance of our method by considering the co-reconstruction of the source points and the $n$-leaf cavity parameterized by
\[
	x_L(t)=(1+0.2\cos nt)(\cos t,\sin t),\quad
\]
where $n\in\mathbb{N}_+$. The wavenumber is taken to be $k=10$ in this example. The radii of the measurement circles are taken to be 0.5 and 0.7. The radii of auxiliary circles are set to be 0.4 and 1.5. We consider three different cavities by modifying $n$. In \cref{fig:starfish1}, we show the model geometry settings together with the reconstructions. In the second column, we depict the images of the indicator function $I_1(y)$ over the sampling domain. It can be seen that the indicator $I_1(y)$ has one significant local maximizer, which matches well with the true source location. The third column shows the final reconstructions subject to $10\%$ noise, where the relative $L^2$ errors for the $n$-leaf cavities are $2.53\% (n=4)$, $1.97\%(n=5)$ and $8.18\%(n=8)$, respectively. These results demonstrate that the proposed method can simultaneously reconstruct the boundaries of the cavities and the source points quantitatively.

\end{example}

\begin{figure}
		\centering	
		\subfigure[]{\includegraphics[width=0.3\textwidth]{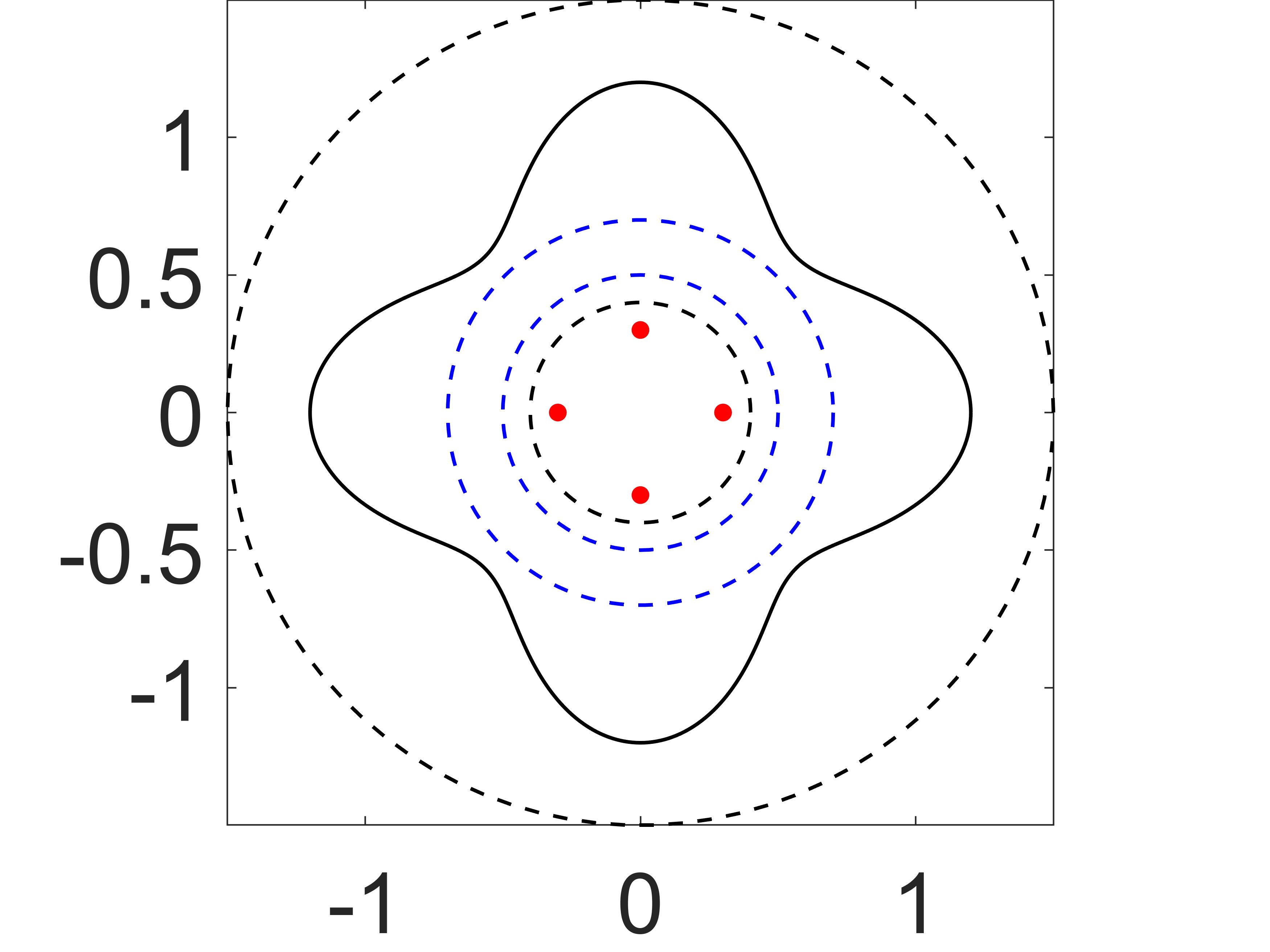}}
		\subfigure[]{\includegraphics[width=0.3\textwidth]{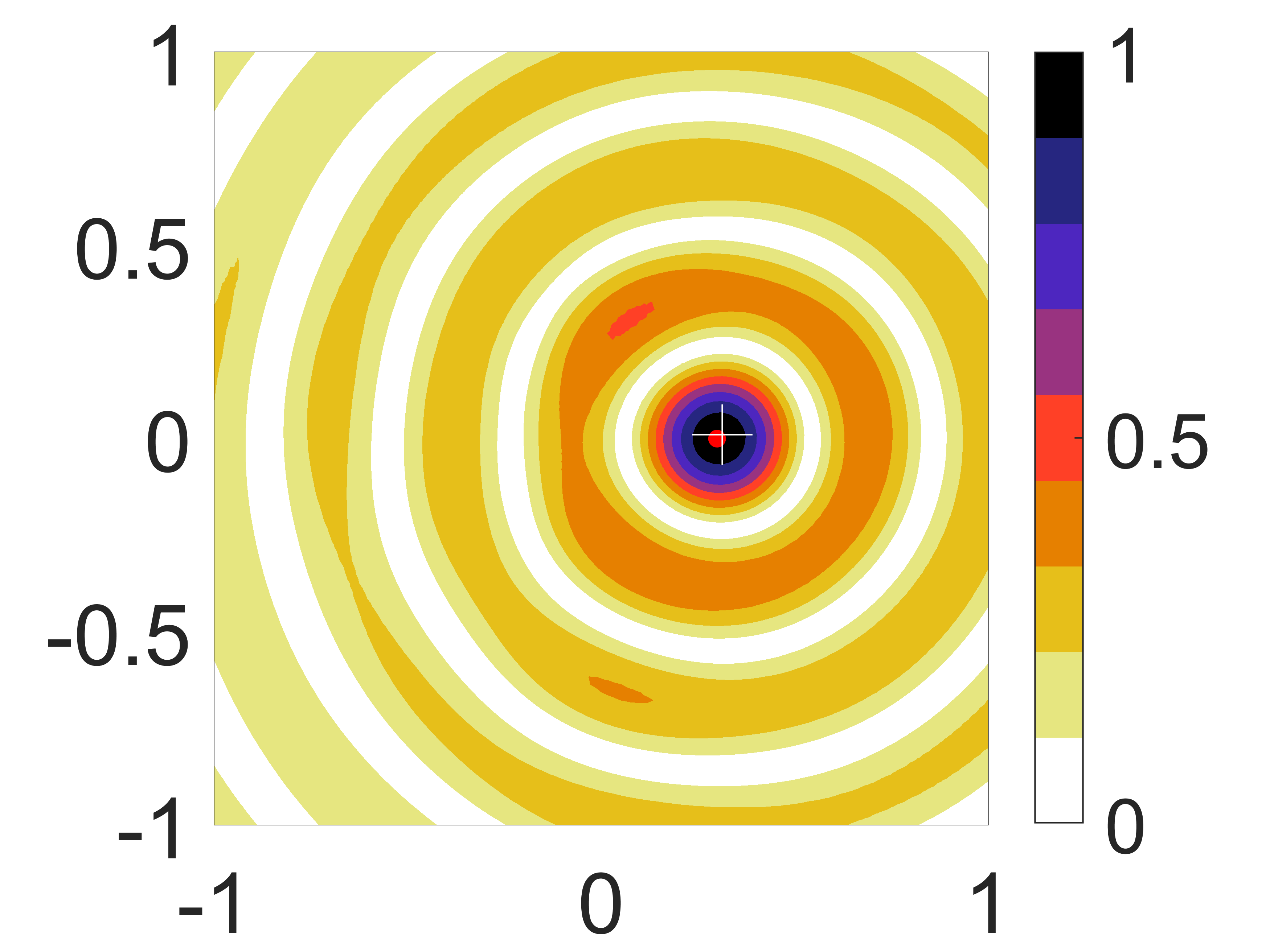}}
		\subfigure[]{\includegraphics[width=0.3\textwidth]{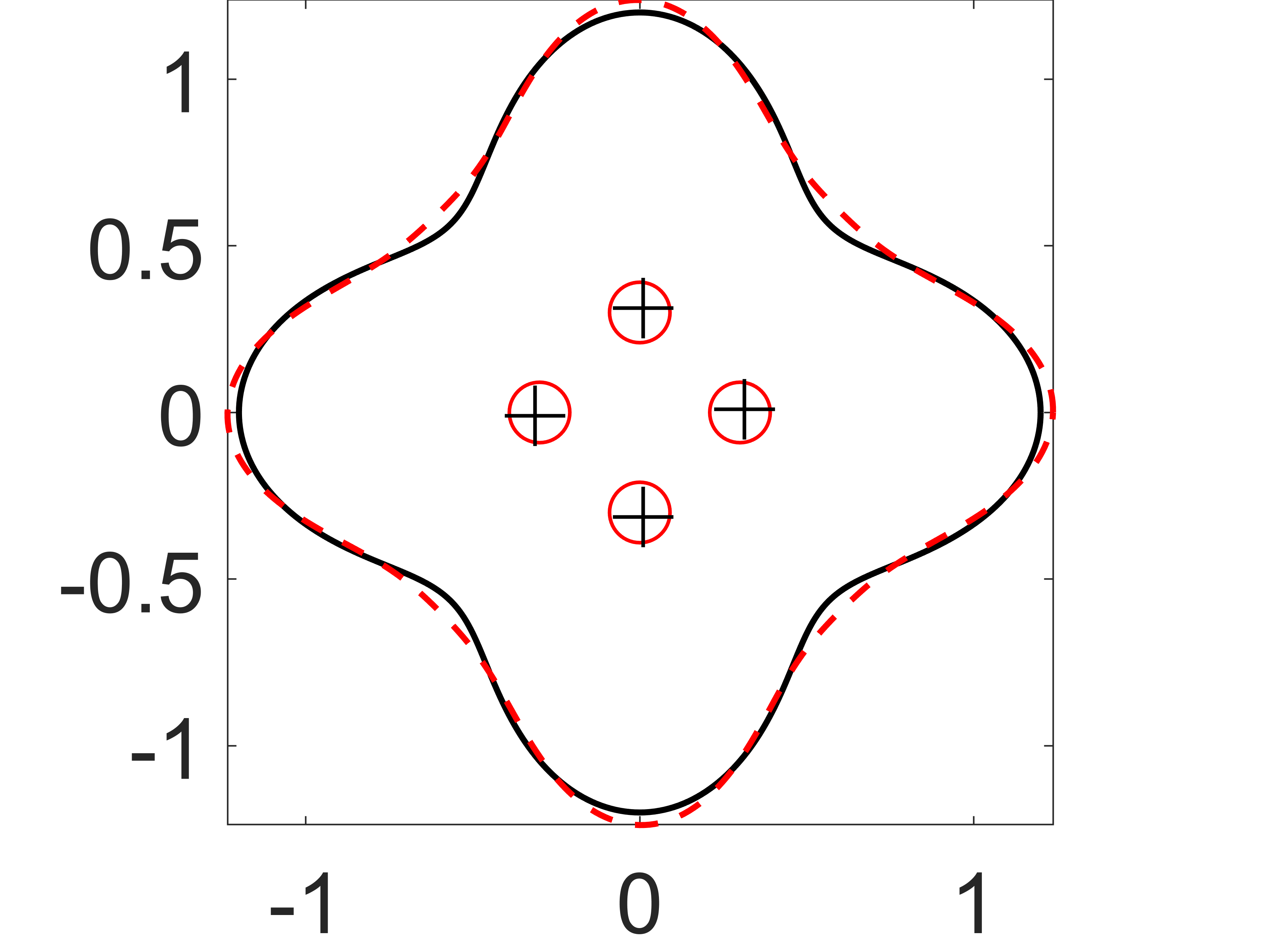}} \\

		\subfigure[]{\includegraphics[width=0.3\textwidth]{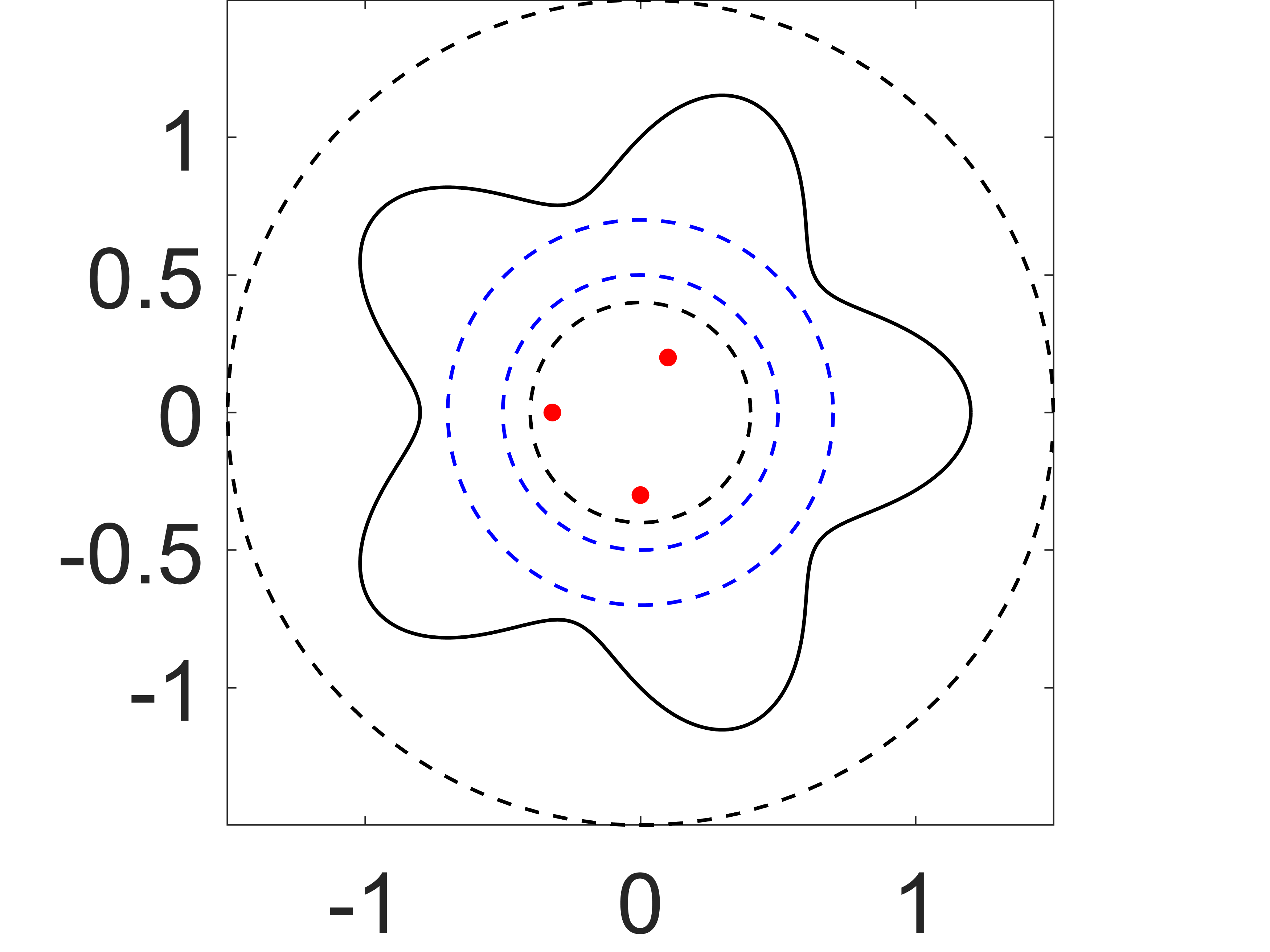}}
		\subfigure[]{\includegraphics[width=0.3\textwidth]{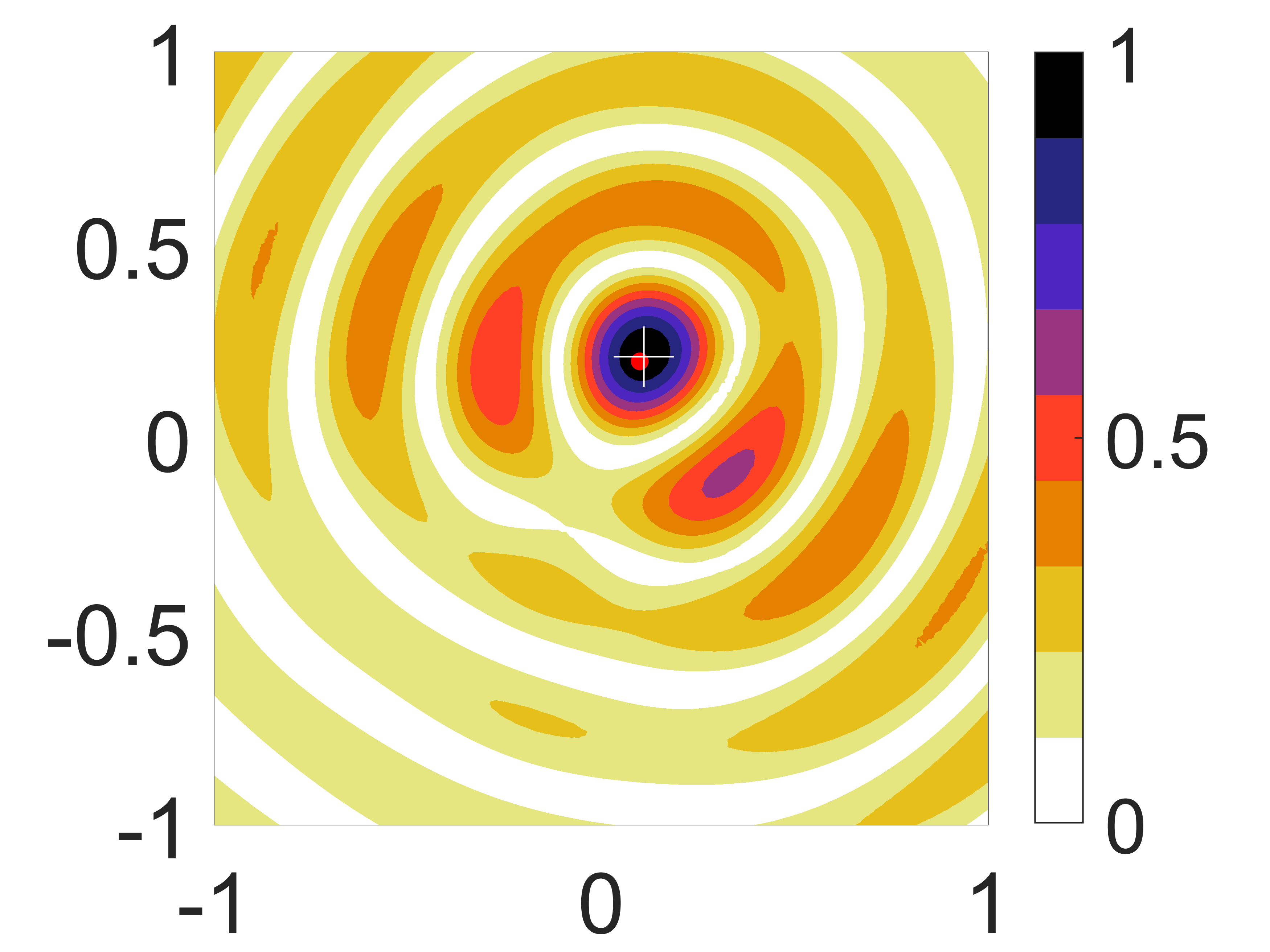}}
		\subfigure[]{\includegraphics[width=0.3\textwidth]{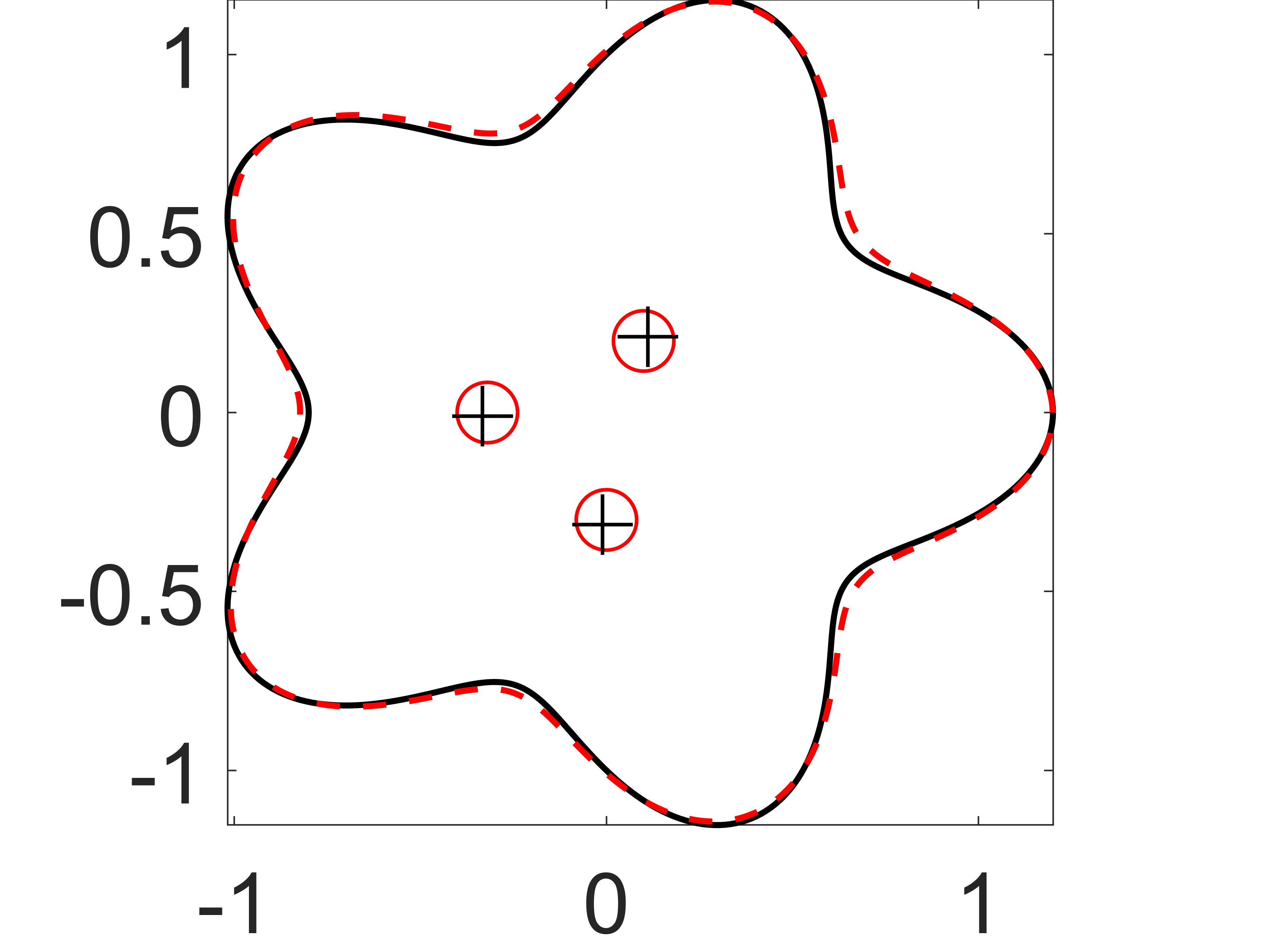}} \\

		\subfigure[]{\includegraphics[width=0.3\textwidth]{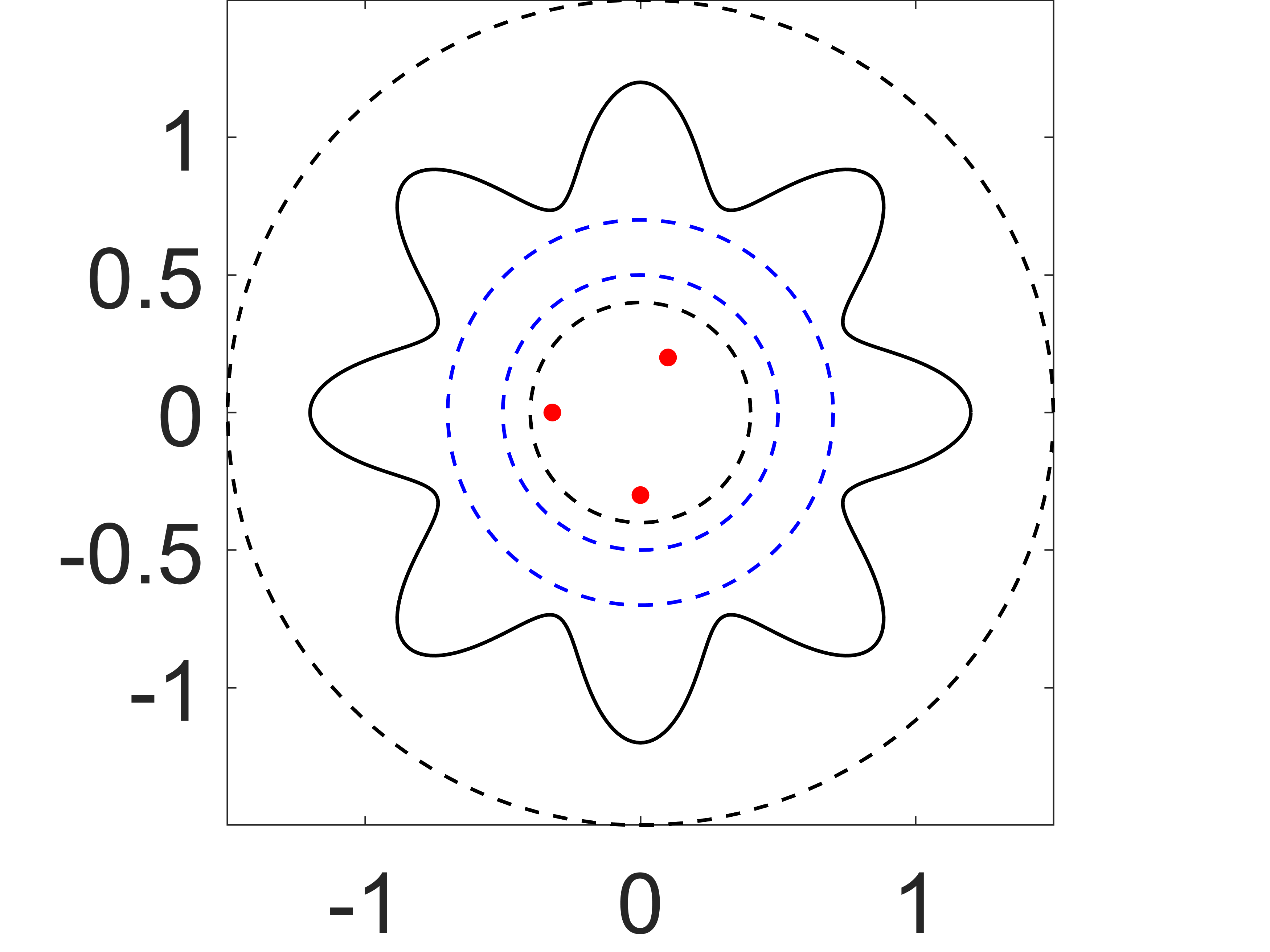}}
		\subfigure[]{\includegraphics[width=0.3\textwidth]{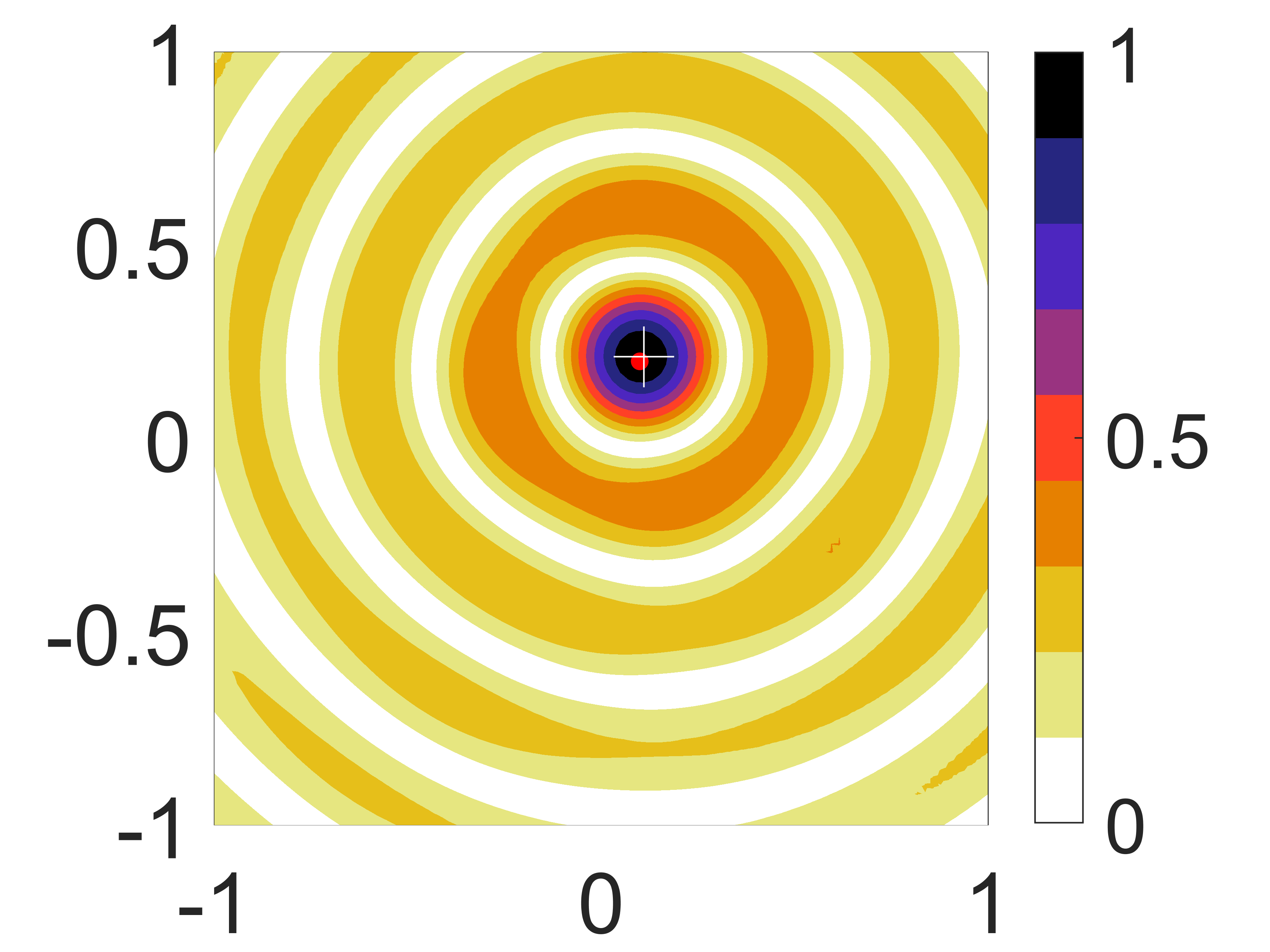}}
		\subfigure[]{\includegraphics[width=0.3\textwidth]{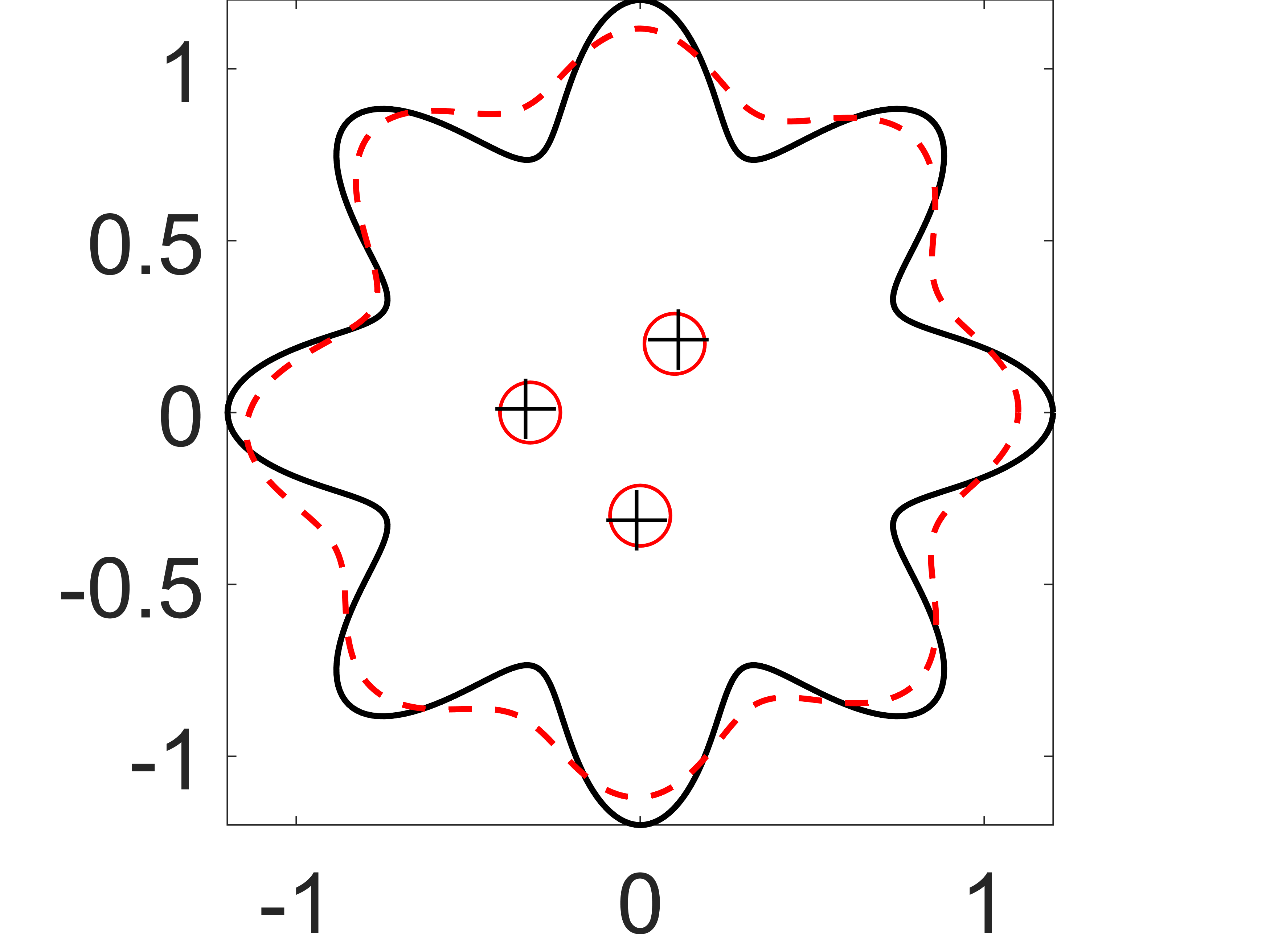}}
		\vskip -0.3cm
		\caption{Reconstruction of the $n$-leaf $(n=4,5,8)$ cavities and the corresponding sources. The first column: Geometry setting; The second column: Images of $I_1(y);$ The third column: Reconstructions.} \label{fig:starfish1}
\end{figure}

\begin{example}
In the second example, we consider the reconstructions of point sources and a kite-shaped cavity parameterized by
\[
		x_K(t)=(\cos t+0.65\cos 2t - 0.65, 1.5\sin t),\quad 0\le t\le 2\pi.
\]
The radii of the measurement circles are taken to be 0.6 and 0.8. The radii of auxiliary circles are 0.5 and 2.2.
		
We first consider the case with $10\%$ noise and 5 point sources. The source points are equally distributed on a circle centered at the origin with radius 0.3. The reconstructions with four different wavenumbers are shown in \cref{fig:kite}. One easily observes from \cref{fig:kite} that our method fails to recover the far-ends of the two wings if the wavenumber is relatively small or large, while the remaining portion could always be well reconstructed.  In fact, this imperfection is typical and reasonable for inverse cavity scattering problems. On one hand, the low-frequency waves generally do not have the capability of resolving fine details. On the other hand, the point sources are clustered around the center of cavity and thus the far-ends are usually  illuminated inadequately because the energy of the wave decays proportionally to the frequency and distance. Similar effects and interpretations can be found in other algorithms for inverse cavity problems, see, e.g., \cite{Guo1}.
	\end{example}

    \begin{figure}
    	\centering	
    	\subfigure[]{\includegraphics[width=0.24\textwidth]{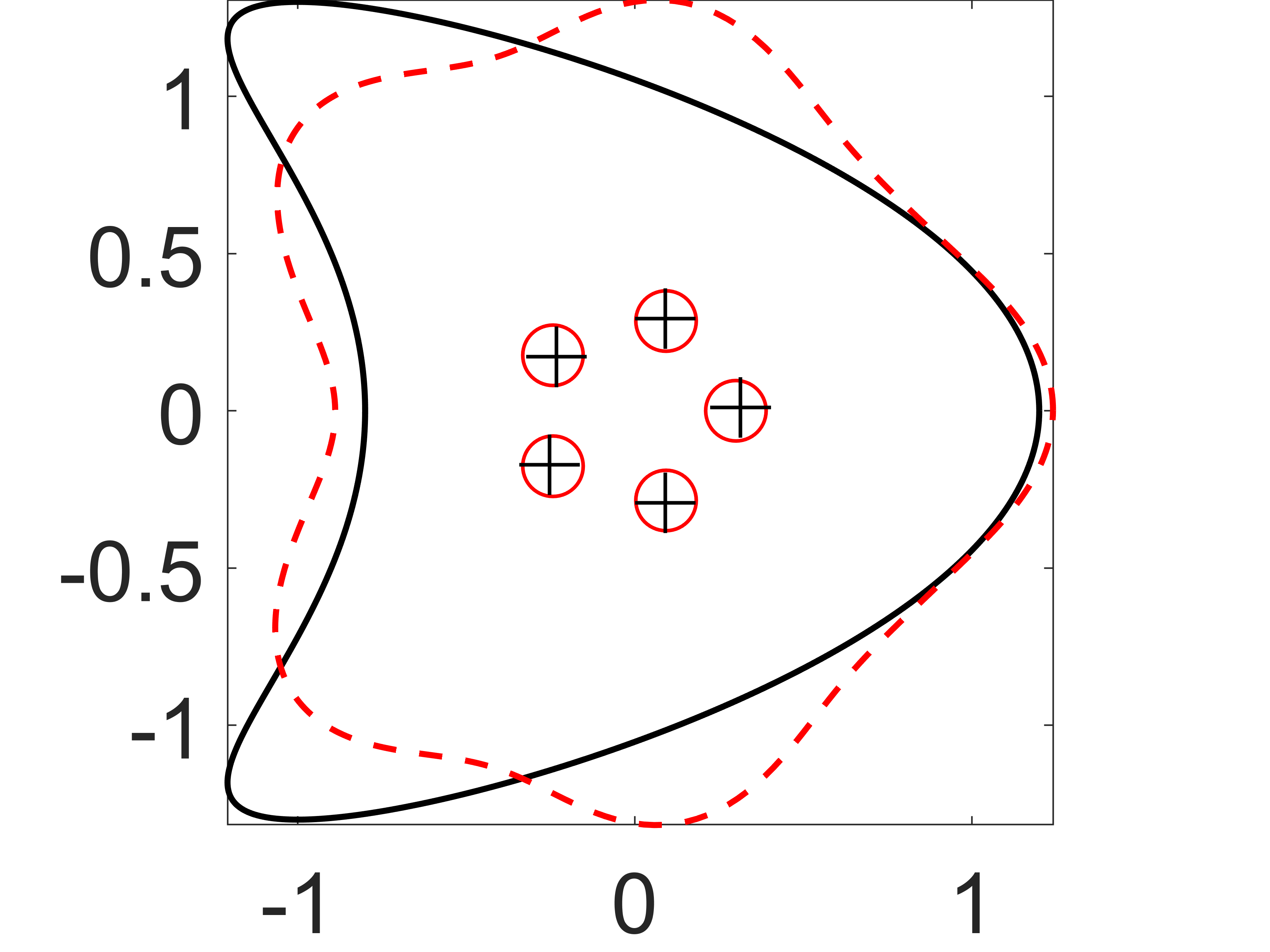}}
    	\subfigure[]{\includegraphics[width=0.24\textwidth]{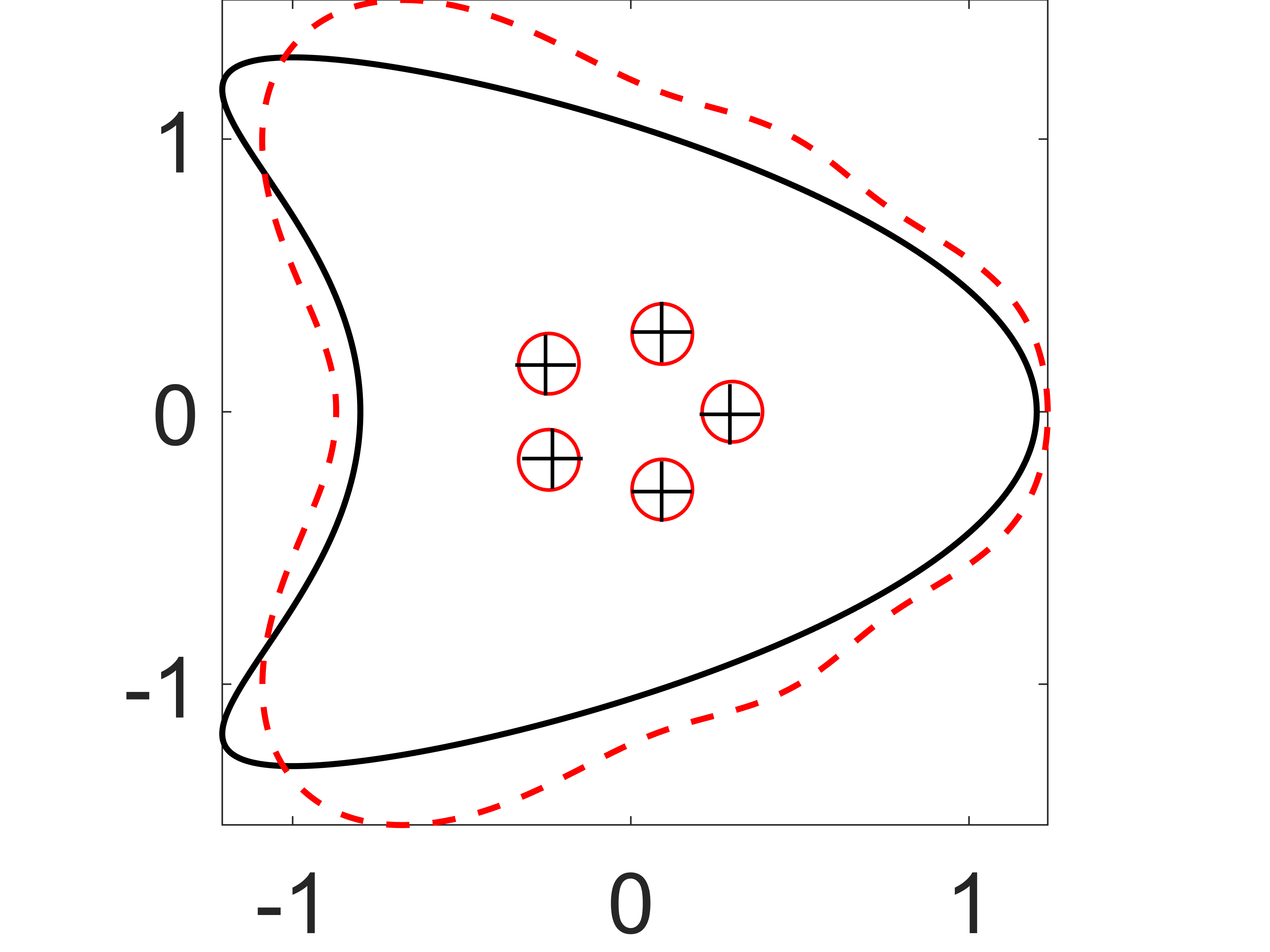}}
    	\subfigure[]{\includegraphics[width=0.24\textwidth]{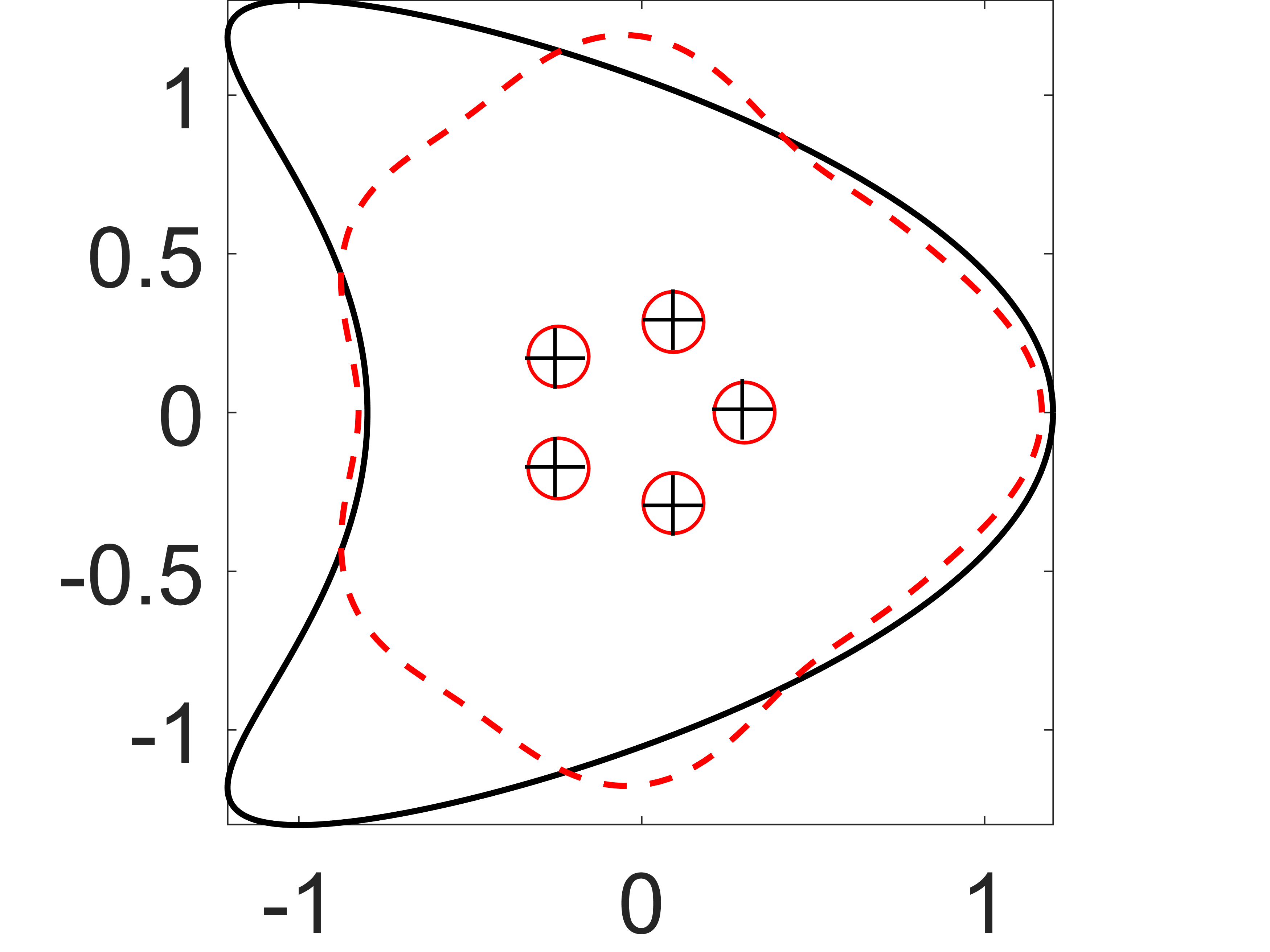}}
    	\subfigure[]{\includegraphics[width=0.24\textwidth]{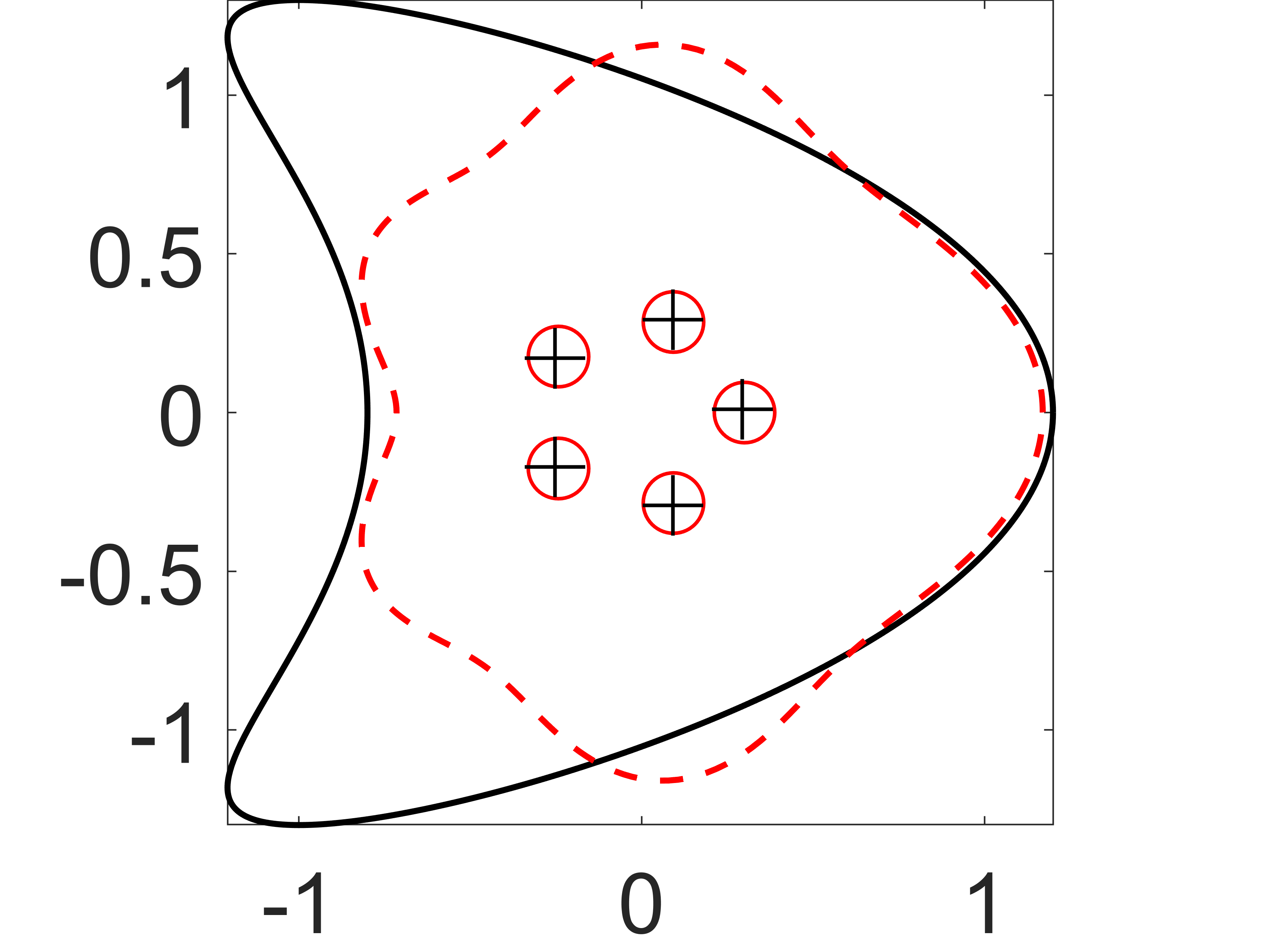}}
    	\caption{Reconstruction of the kite-shaped cavity and source points with different wavenumbers (from left to right: $k=2, 4, 6, 10$).}
    	\label{fig:kite}
    \end{figure}

Next, we fix the wavenumber to be $k=4$ and investigate the reconstruction with respect to the number and distribution of   source points.  The corresponding reconstructions are depicted in \cref{fig:kite1}, which illustrates that the distribution of the source points has an influence on the reconstruction of the cavity. In other words, the two components in the co-inversion problem, namely the inverse source problem and inverse scattering problem, are closely interlinked.

   \begin{figure}
	    \centering	
	    \subfigure[]{\includegraphics[width=0.3\textwidth]{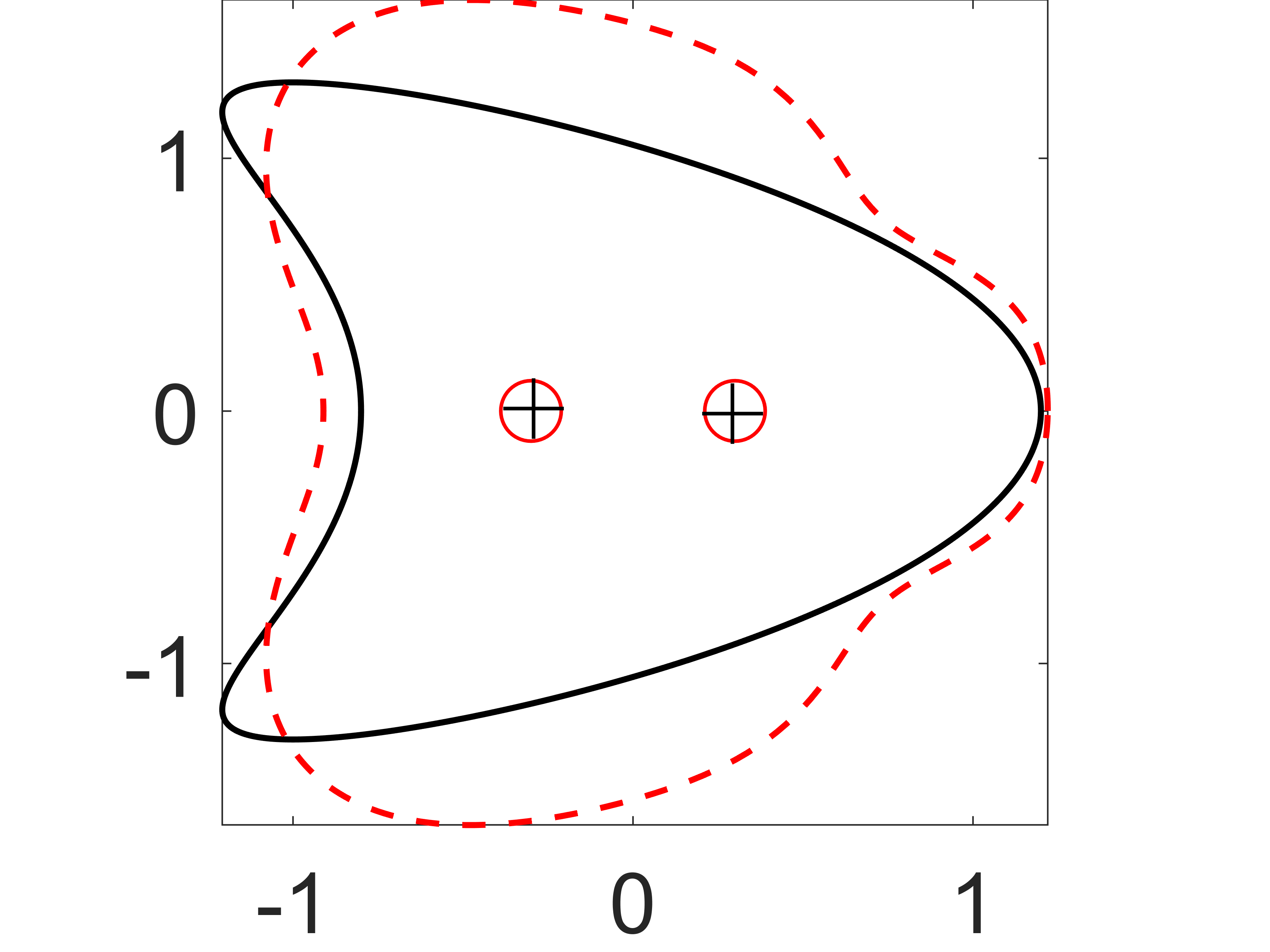}}
	    \subfigure[]{\includegraphics[width=0.3\textwidth]{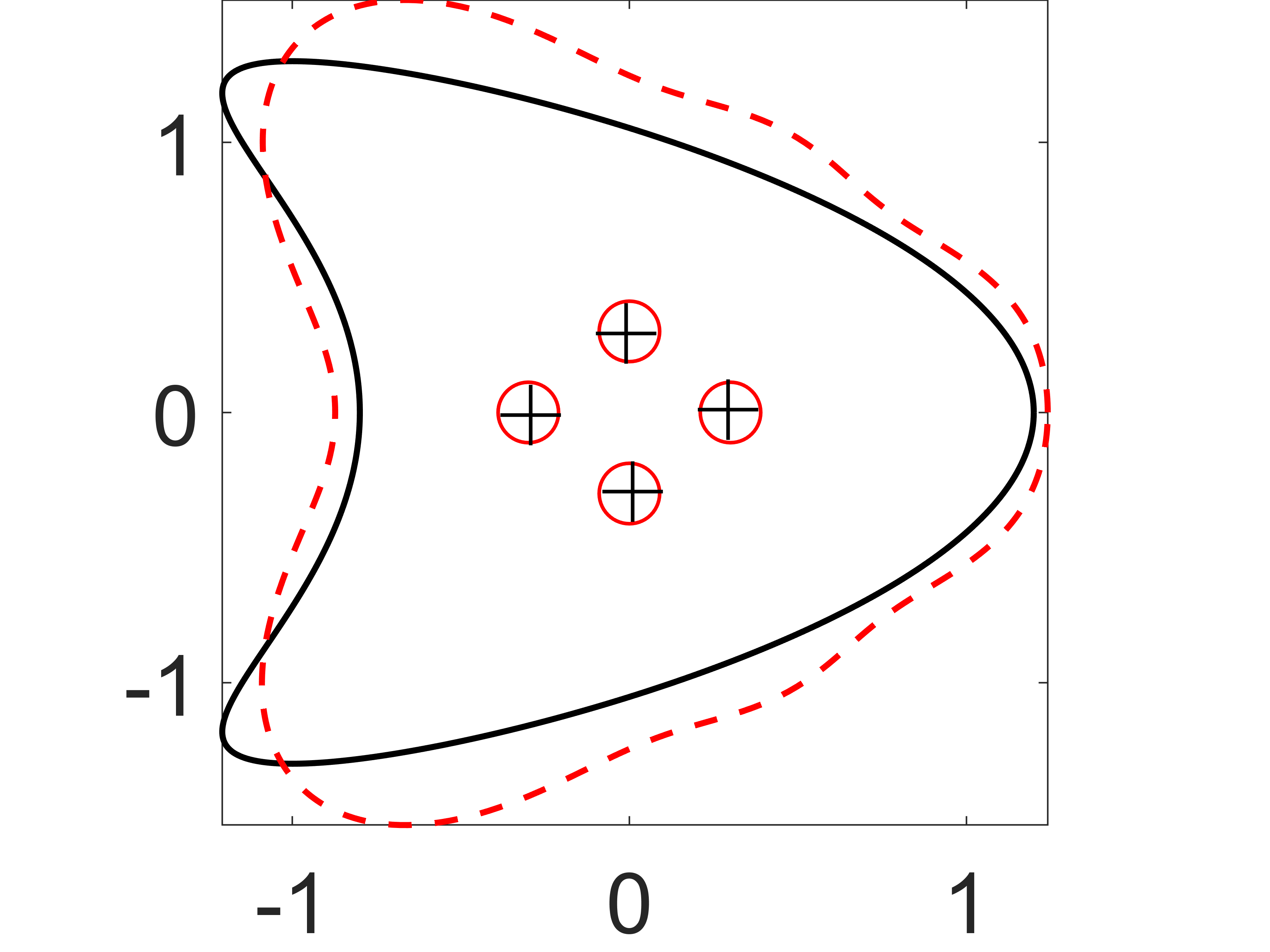}}
	    \subfigure[]{\includegraphics[width=0.3\textwidth]{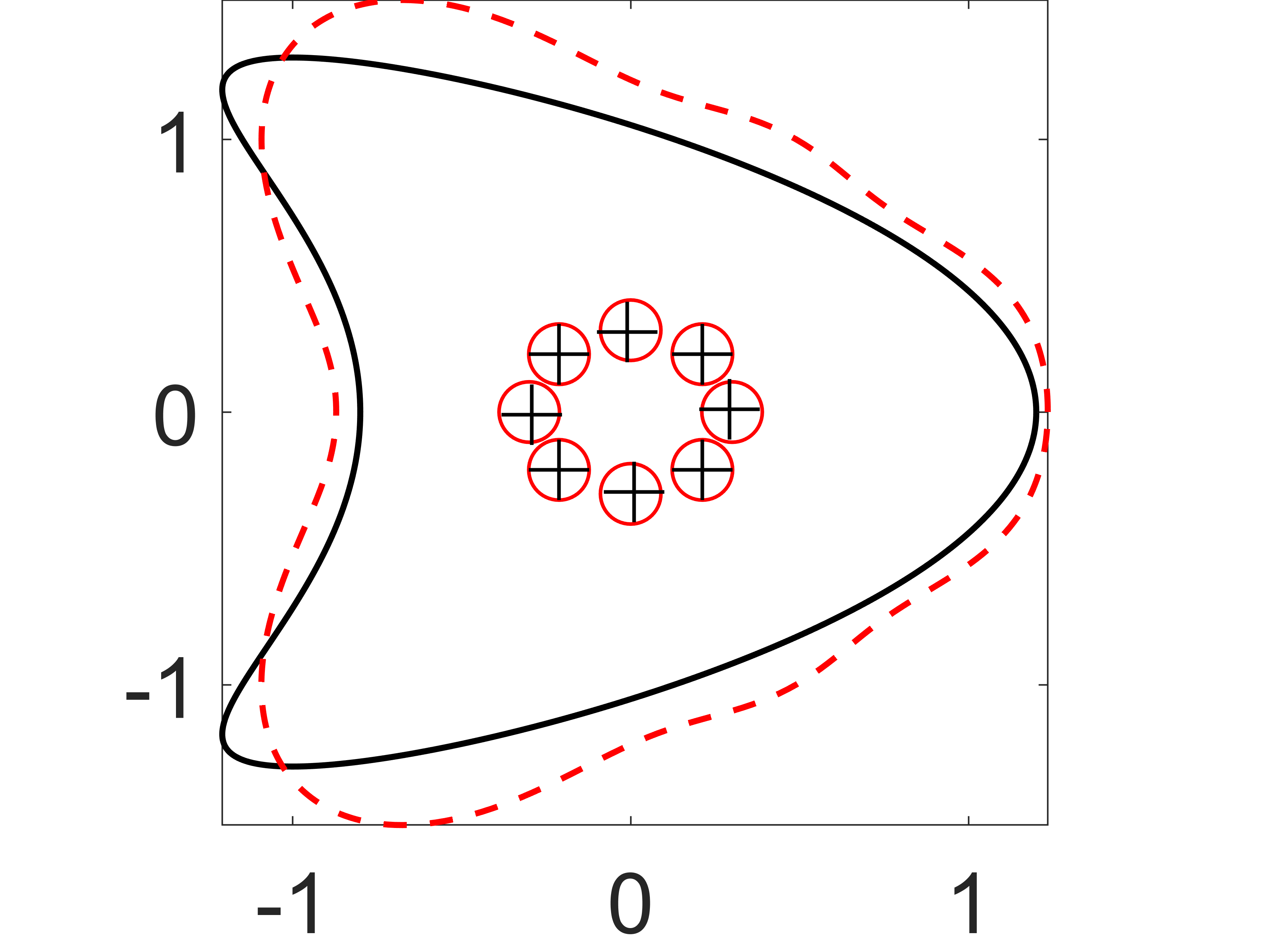}}
	    \subfigure[]{\includegraphics[width=0.3\textwidth]{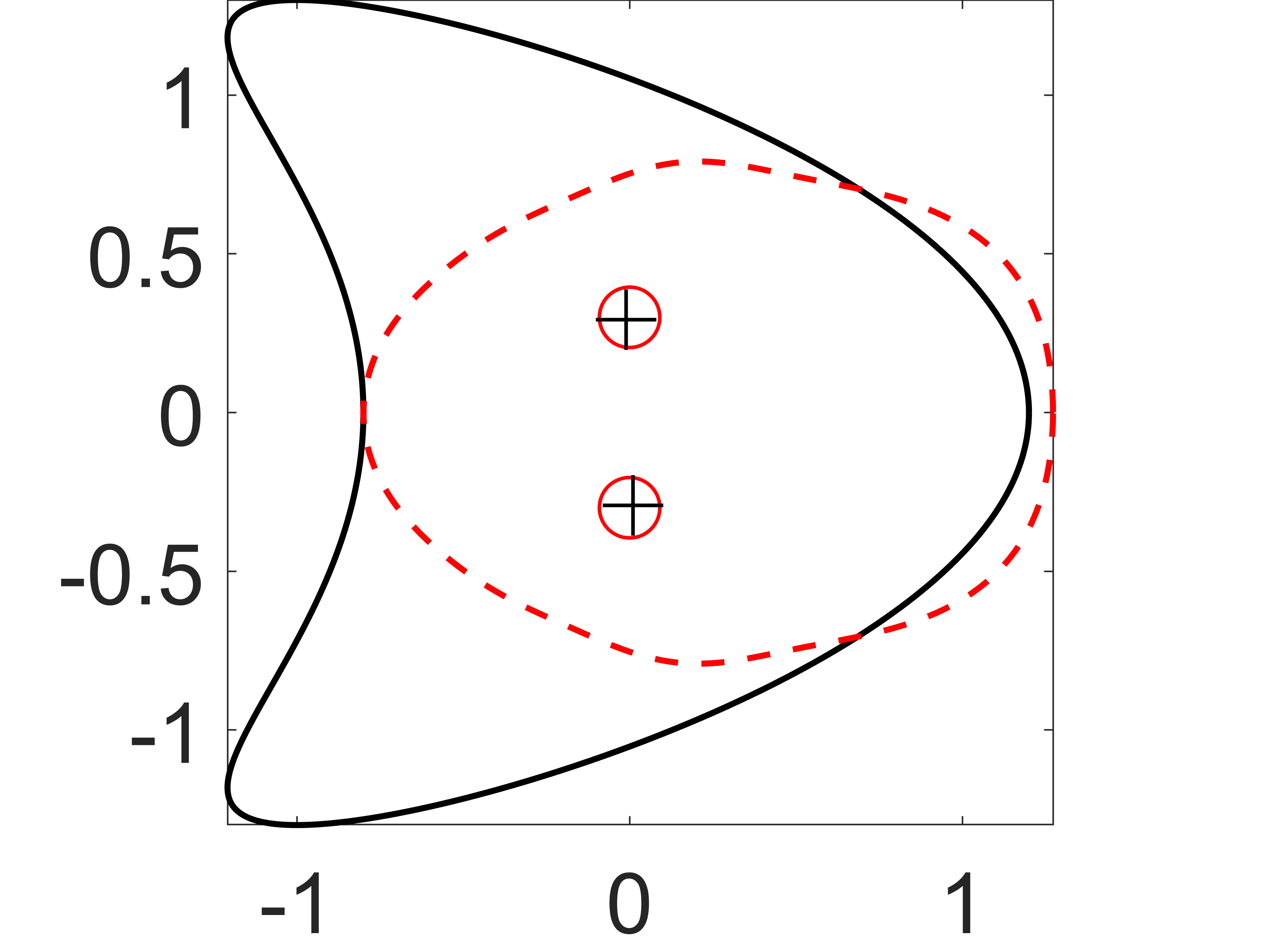}}
	    \subfigure[]{\includegraphics[width=0.3\textwidth]{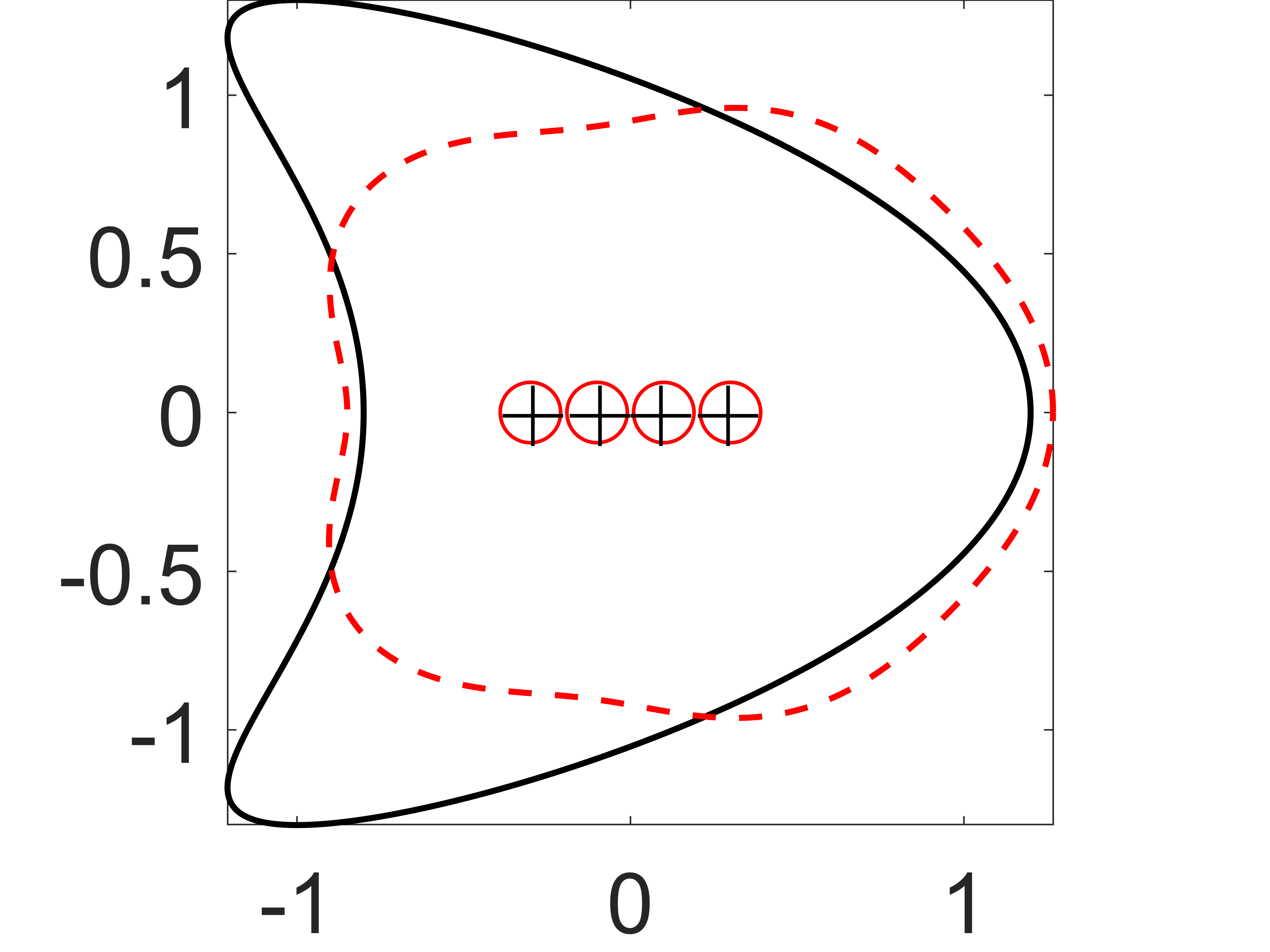}}
	    \subfigure[]{\includegraphics[width=0.3\textwidth]{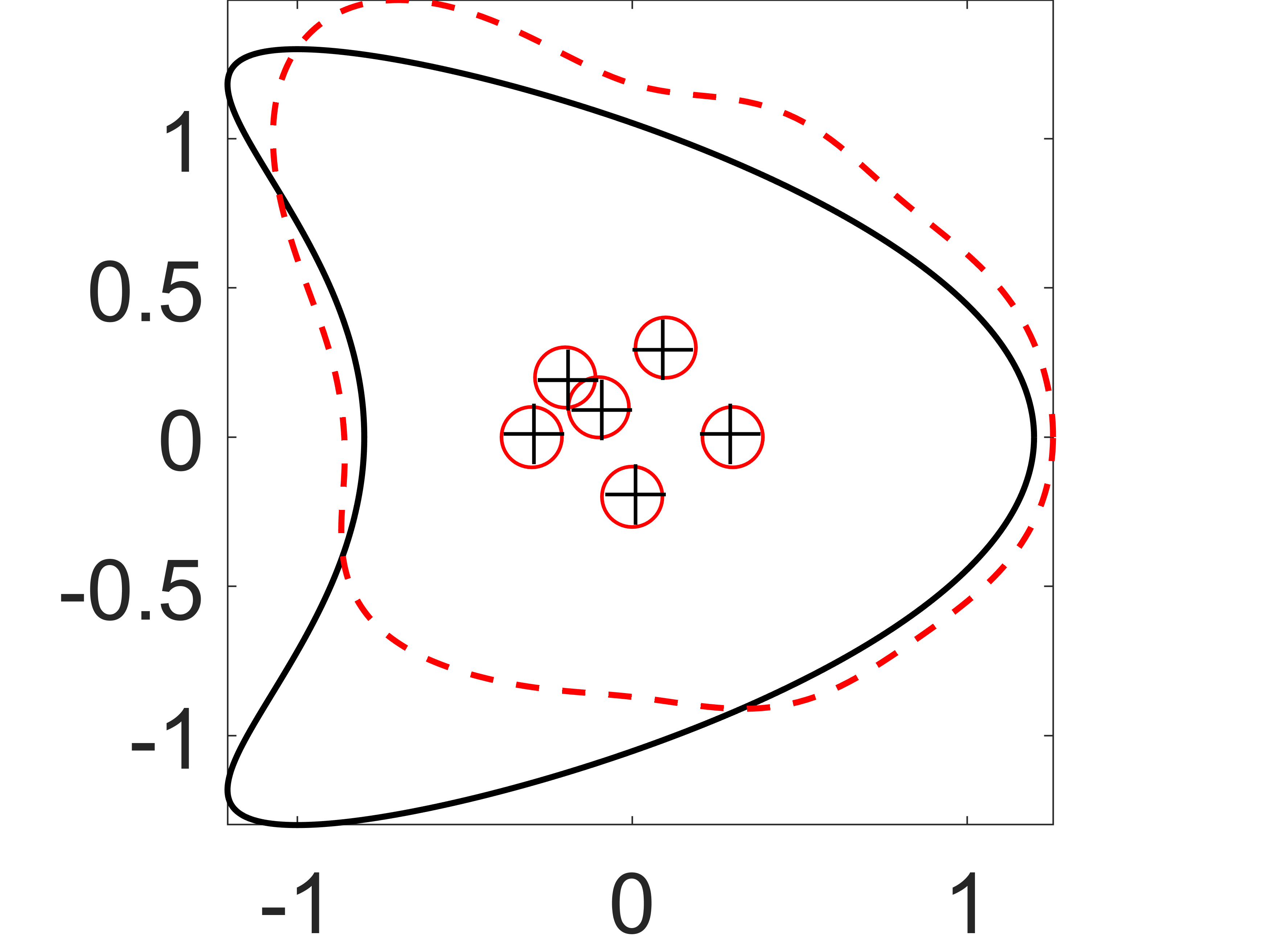}}
	    \caption{Reconstruction of the kite and different number and distribution of source points.}
	    \label{fig:kite1}
   \end{figure}

\begin{example}
	In the last example, we consider the simultaneous reconstruction of the source points together with non-symmetric cavities. The boundary curves of the cavities are described in the form of
	\begin{equation*}
			x(t)=r(t)(\cos t,\sin t),\quad t\in[0, 2\pi),
    \end{equation*}
    where the radial function $r(t)$ is randomly generated by the following procedure:
	    \begin{itemize}
            \item Introduce a set of equidistant knots in $[0,2\pi]$ by $T_t=2\pi t/n_T,\,t=0,1,\cdots,n_T-1$ with $n_T$ the number of knots chosen randomly in the integer set $\{8,9,\cdots,16\};$
	    	\item  For each $T_t$, take the corresponding radial gird knot $r(T_t)$ to be a uniform distribution on interval $[0.4, 1.6]$, that is,  $r(T_t)\sim\mathcal{U}[0.4, 1.6]$. In this setting, the random curves are produced in the annular domain with inner radius 0.4 and outer radius 1.6.
	    	\item Generate the radial function $r(t)$ by the cubic spline interpolation to the given data $(T_t, r(T_t)),$ $t=0, 1,\cdots,n_T-1$; We also impose the periodic condition $r(T_0)=r(T_{n_T})$ to guarantee that the star-like curves are closed;
	    	\item  Take 128 equidistant gird points on $[0, 2\pi],$ with $\theta_i=2\pi i/128,i=0,\cdots,127,$ and the points $(\theta_i, r(\theta_i))$ formulate our boundary curve, then the random shapes are generated.
    	\end{itemize}

    We refer to \cref{fig:illustrate} for an illustration of three boundary curves generated by such procedure, where the red small points denote the $n_T$ interpolation knots. The black solid curves are the boundaries generated by the cubic spline interpolation. The blue circles with radii 0.4 and 1.6 bound the region containing the random curves.

    \begin{figure}
    	\centering	
    	\subfigure[]{\includegraphics[width=0.3\textwidth]{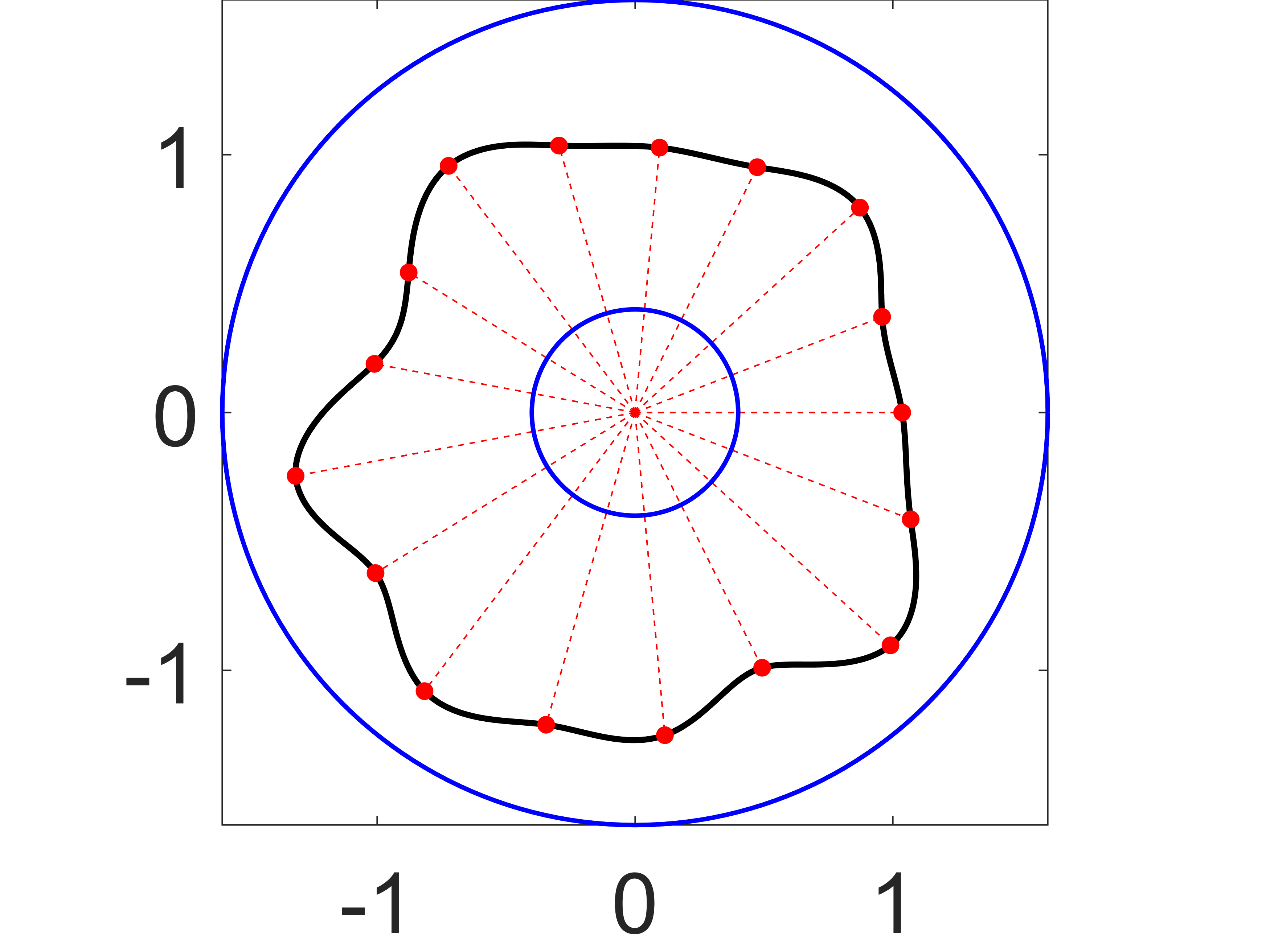}}
    	\subfigure[]{\includegraphics[width=0.3\textwidth]{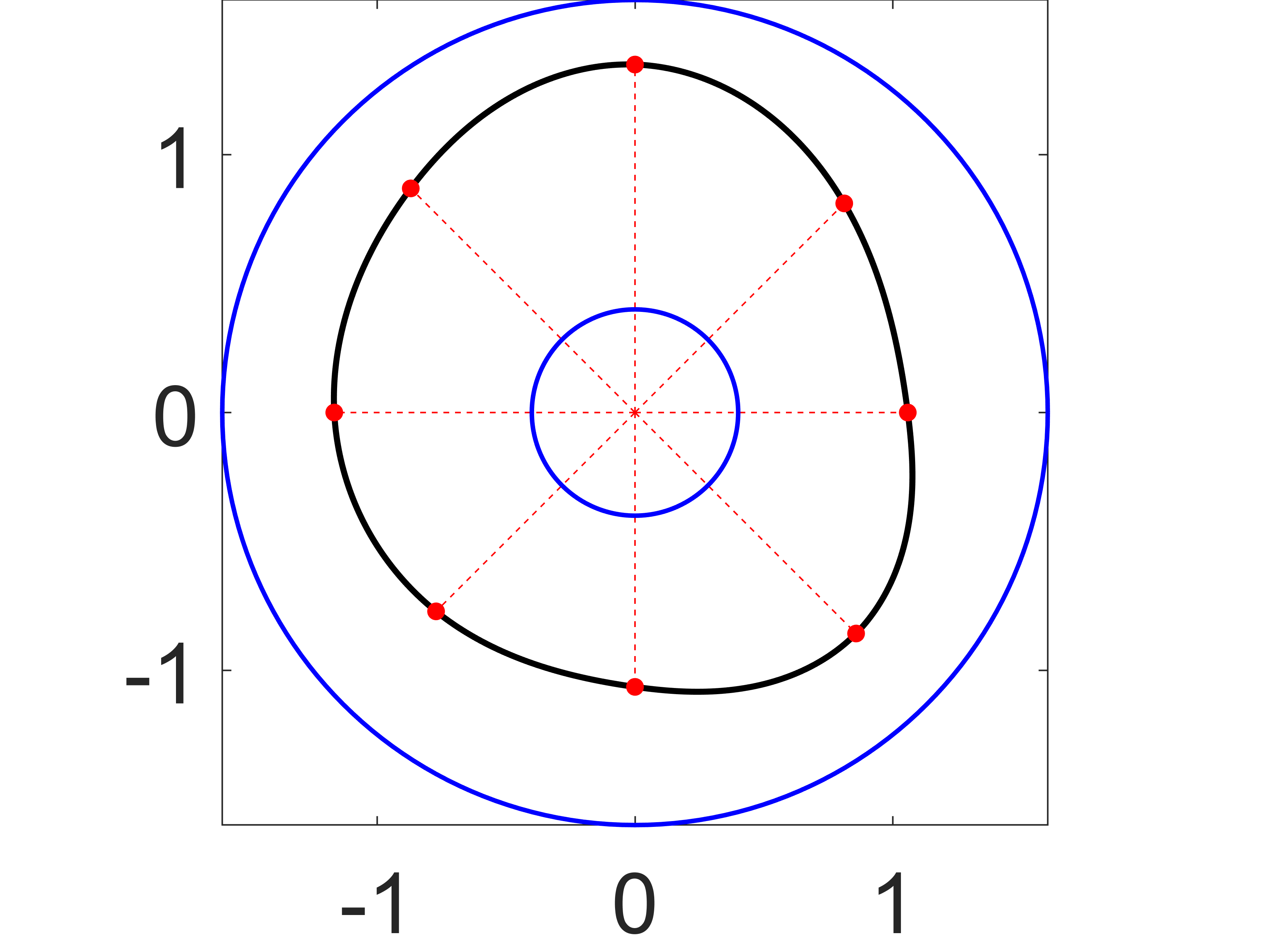}}
    	\subfigure[]{\includegraphics[width=0.3\textwidth]{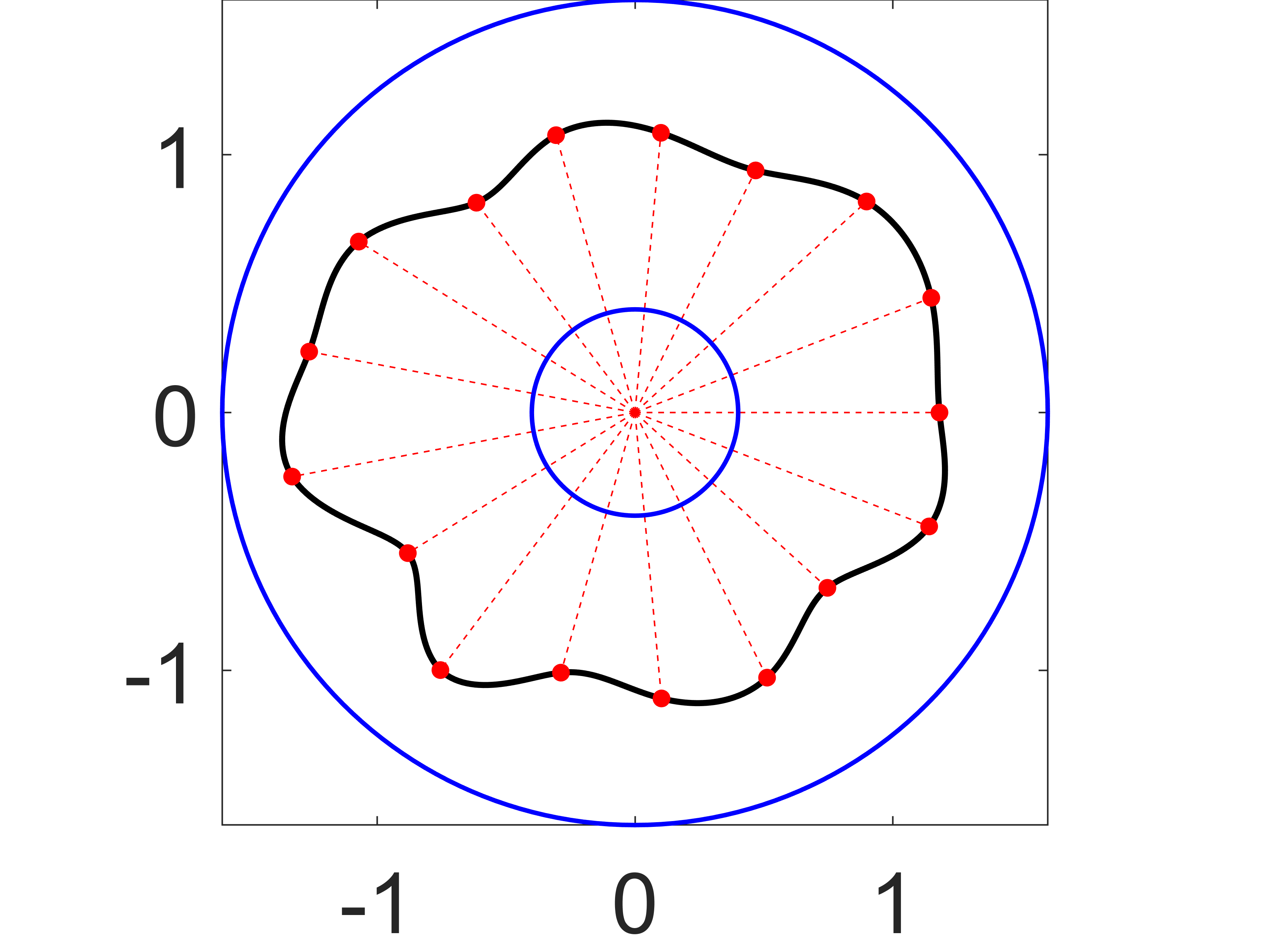}}
    	\caption{An illustration of several random shapes.}
    	\label{fig:illustrate}
    \end{figure}

In this example, the wavenumber $k=4$ and noise level $\delta=10\%$ are used.  The measurements are taken on the concentric circles with radii 0.5 and 0.8. The radii of auxiliary circles are set to be 0.4 and 1.5.

In \cref{fig:random}, we illustrate the reconstructions with point sources located at $(0.1,0.2), (-0.3,0)$ and $(0,-0.32)$. We also consider the cases with more point sources in \cref{fig:random2}. Further, we list the errors of the cavity reconstructions in \cref{tab:error_random} and a comparison of source reconstruction in \cref{tab:location}.
All these numerical results clearly demonstrate that the inversion algorithm performs well in identifying the locations of the point sources. Meanwhile, the randomly generated cavities can be also satisfactorily reconstructed by our method. Moreover, it can be observed that the proposed two-phase (decoupling-imaging) scheme is robust in the sense that the performance is insensitive to the initial guess of  the cavity and the number of point sources.

	\end{example}
	
	\begin{figure}
		\centering	
		\subfigure[]{\includegraphics[width=0.3\textwidth]{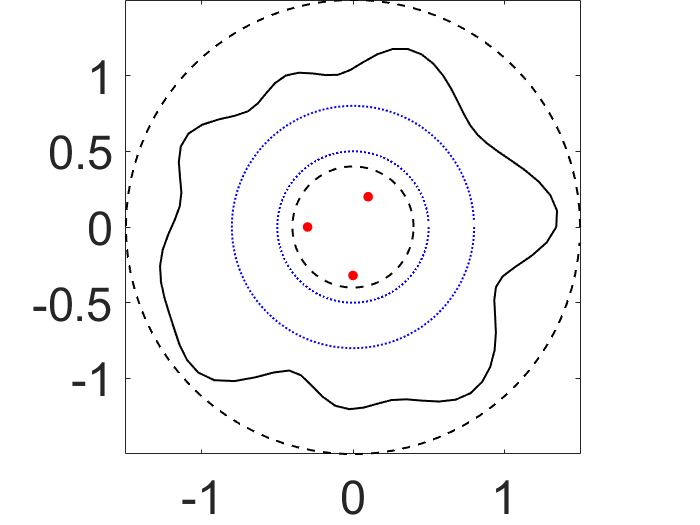}}
		\subfigure[]{\includegraphics[width=0.3\textwidth]{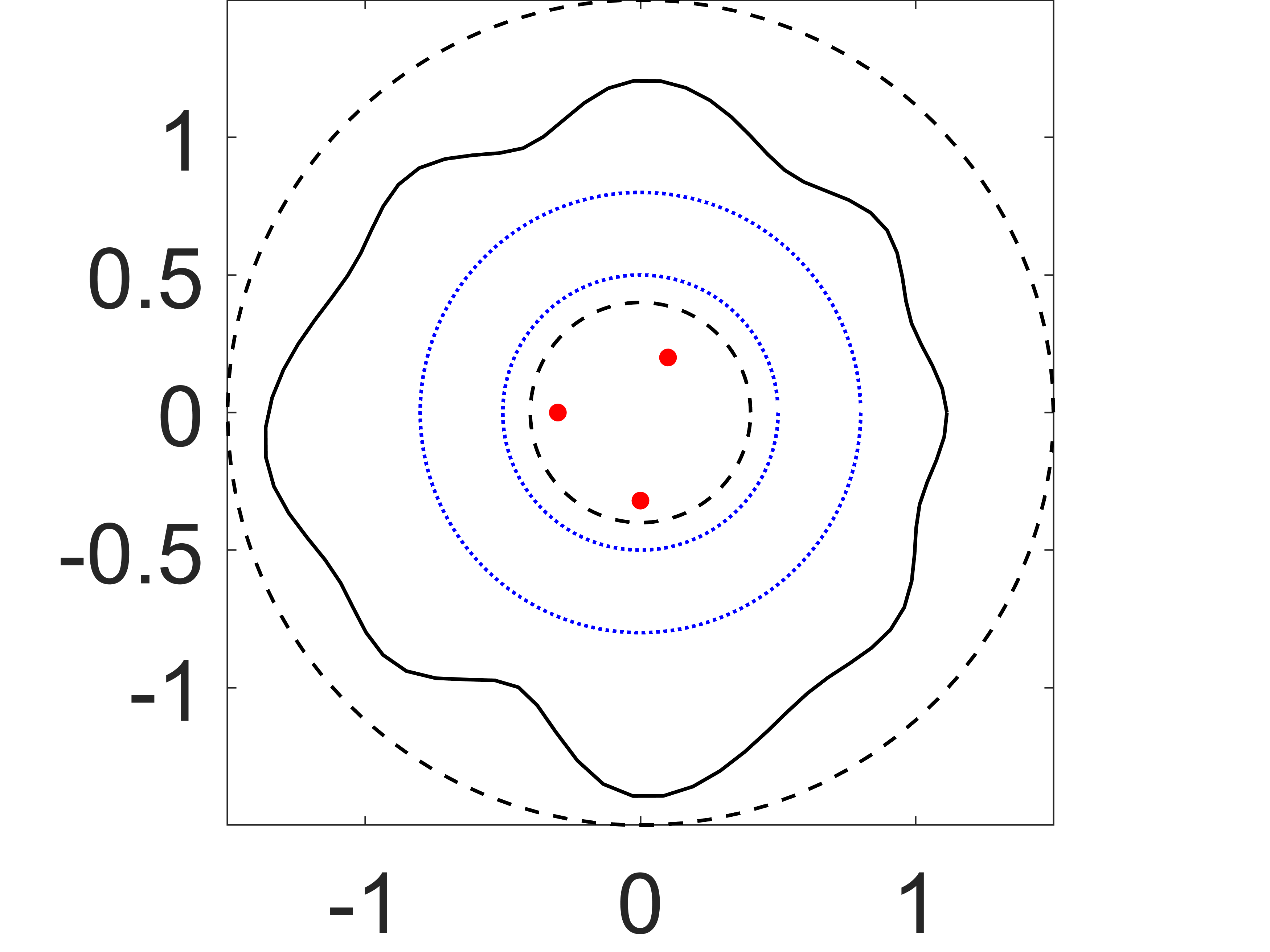}}
		\subfigure[]{\includegraphics[width=0.3\textwidth]{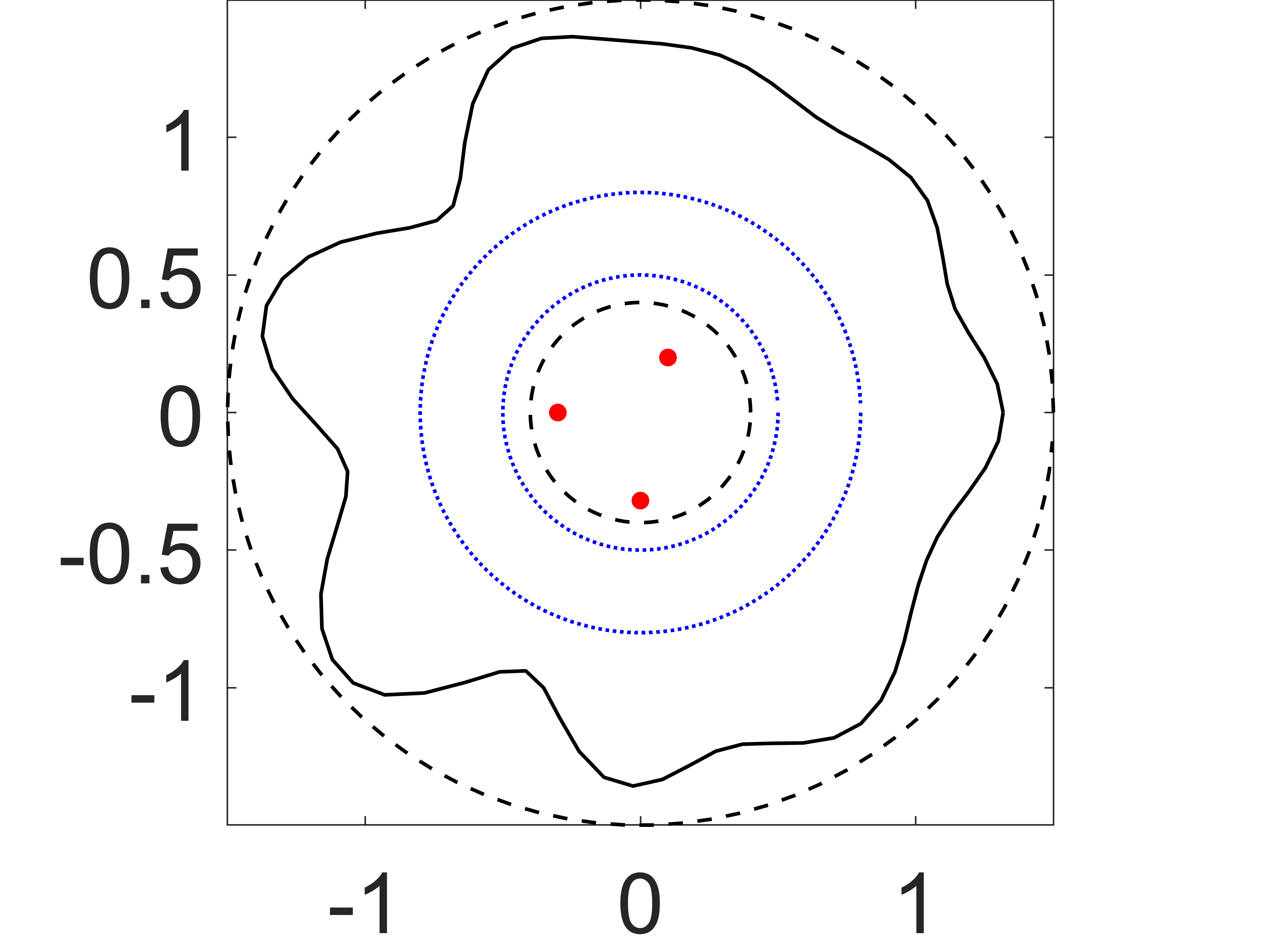}}
		\subfigure[]{\includegraphics[width=0.3\textwidth]{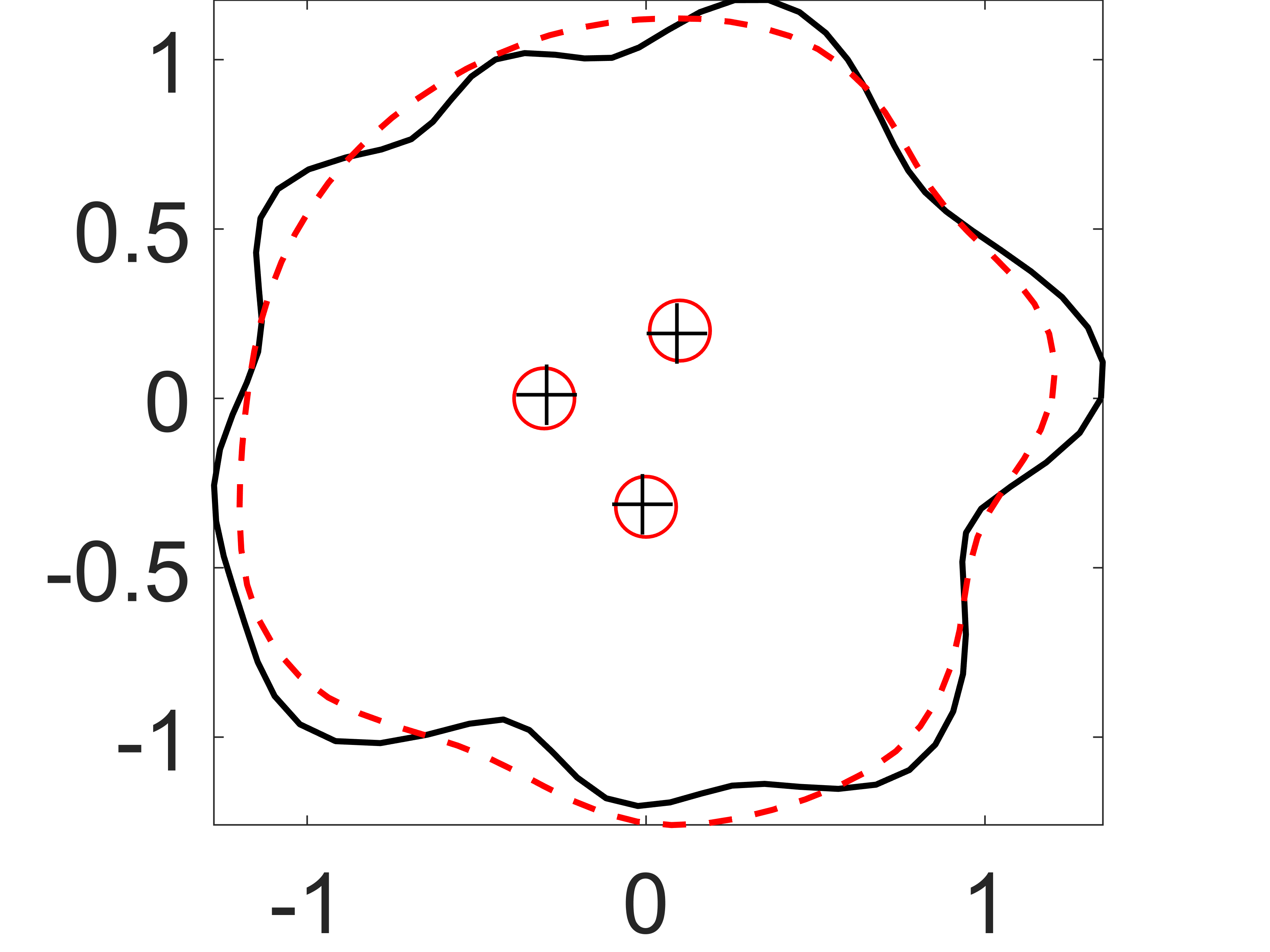}}
		\subfigure[]{\includegraphics[width=0.3\textwidth]{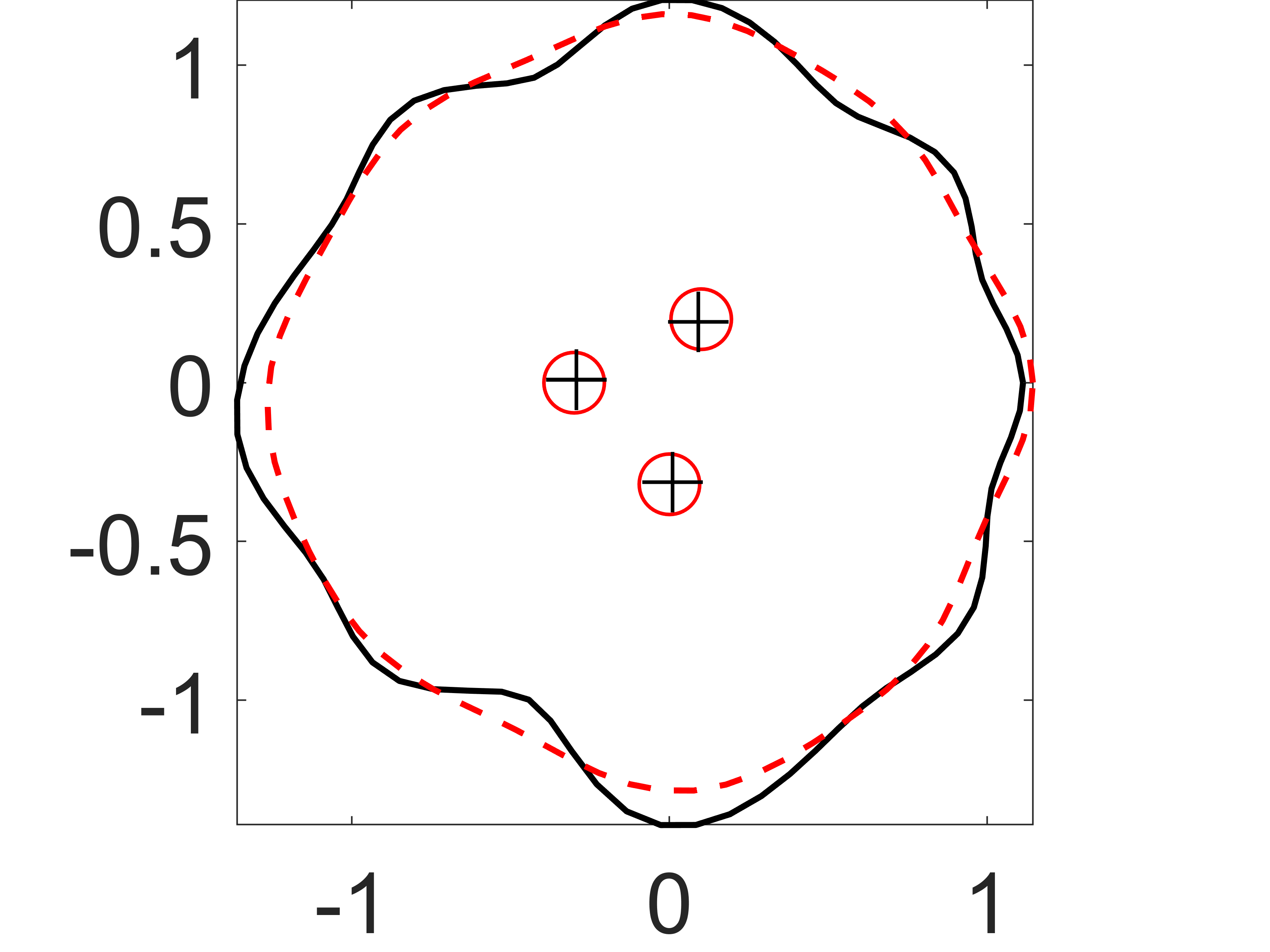}}
		\subfigure[]{\includegraphics[width=0.3\textwidth]{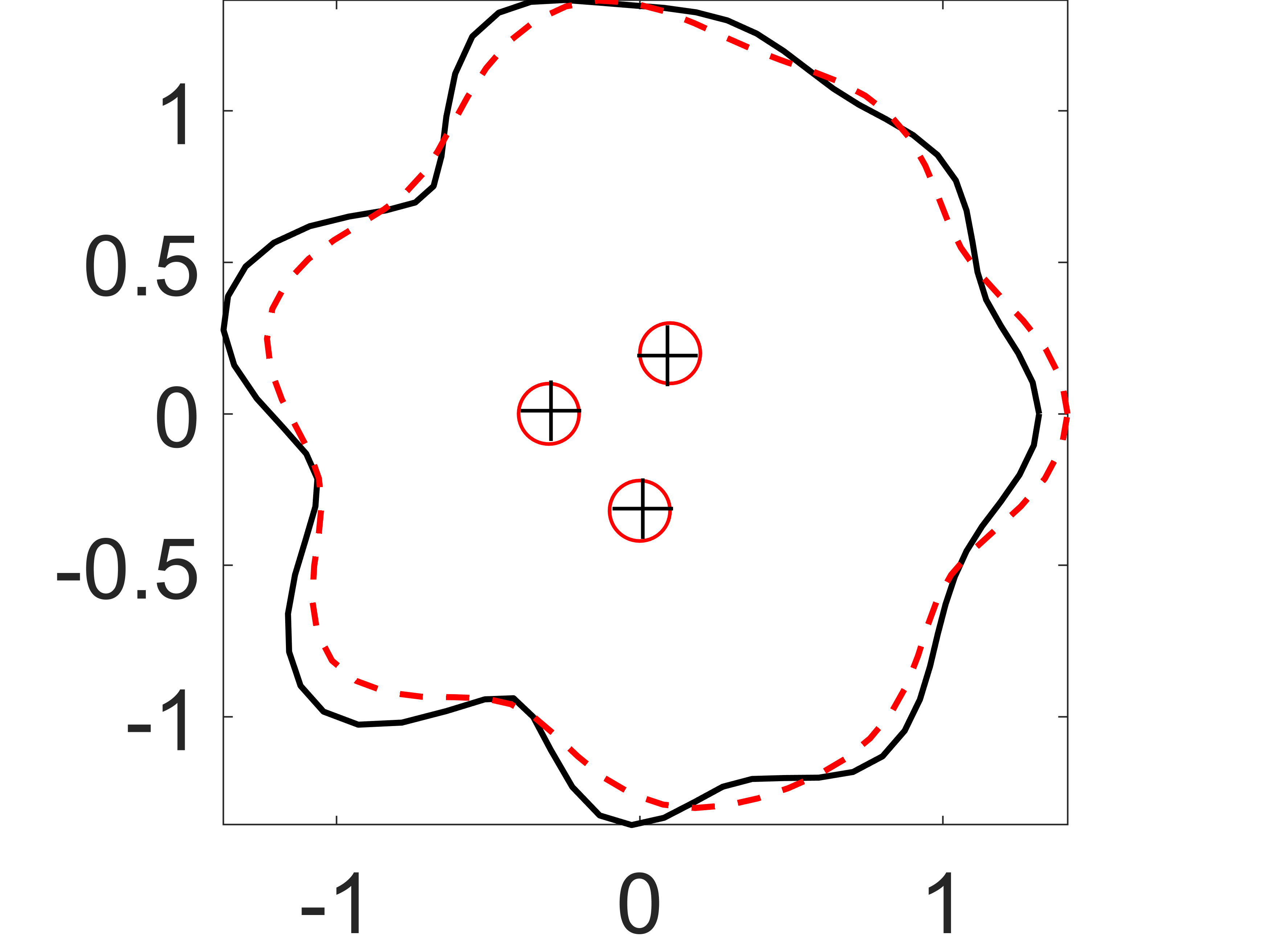}}
		\caption{Reconstruction of the non-symmetric cavities and three fixed source points. Row 1: configurations; Row 2: reconstructions.}
		\label{fig:random}
	\end{figure}

	\begin{figure}
	    \centering	
	    \subfigure[]{\includegraphics[width=0.3\textwidth]{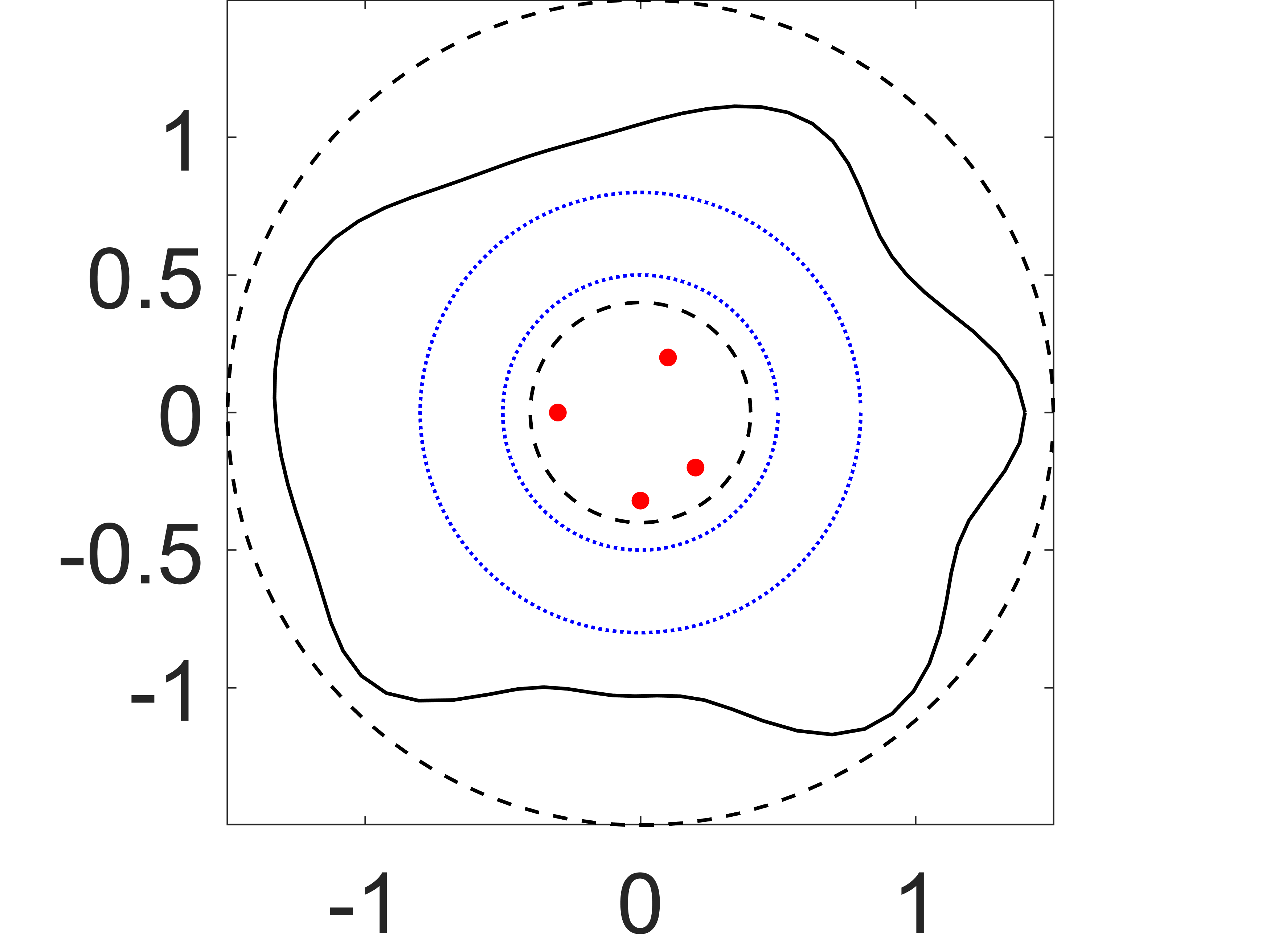}}
	    \subfigure[]{\includegraphics[width=0.3\textwidth]{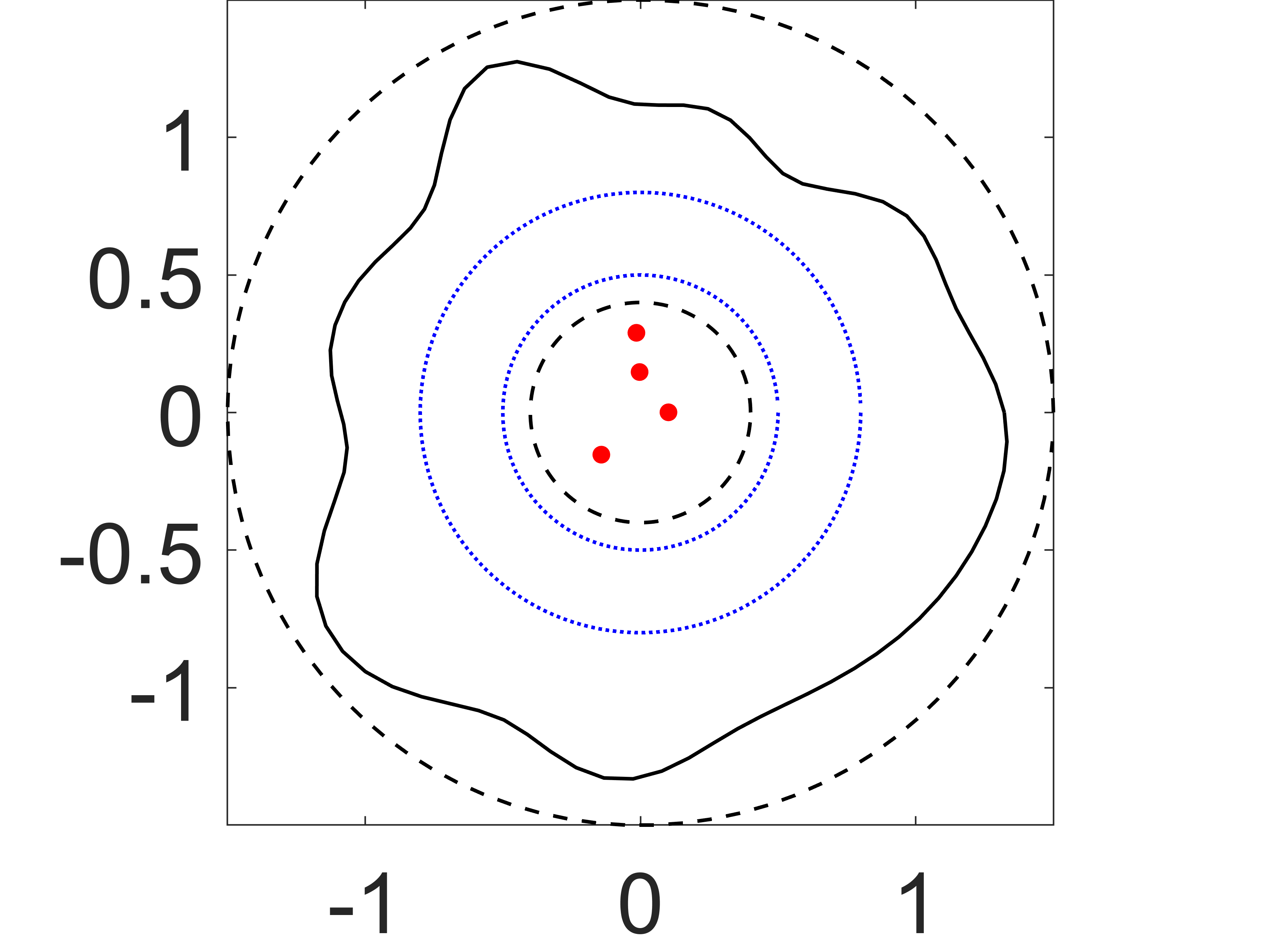}}
	    \subfigure[]{\includegraphics[width=0.3\textwidth]{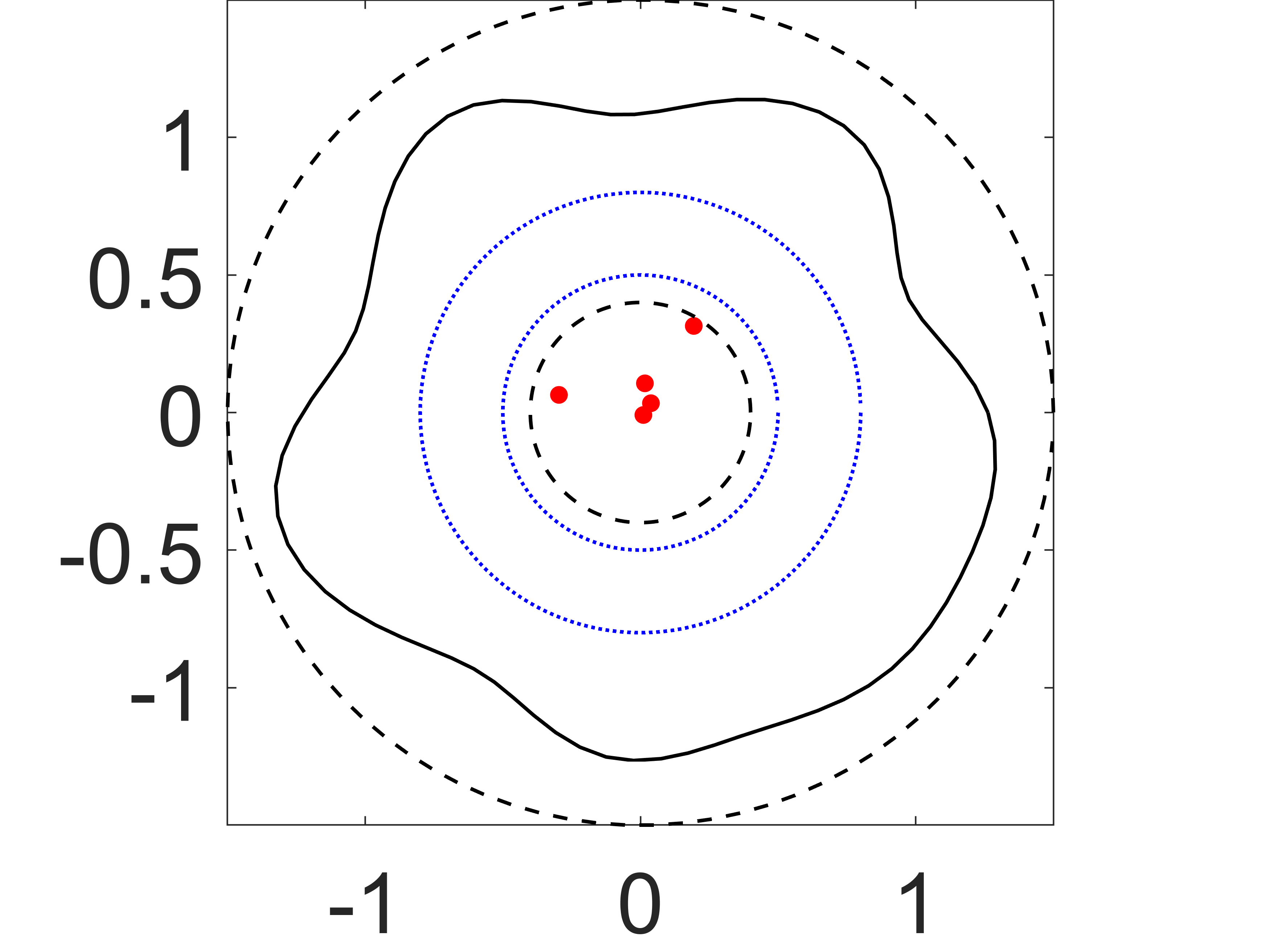}}
	    \subfigure[]{\includegraphics[width=0.3\textwidth]{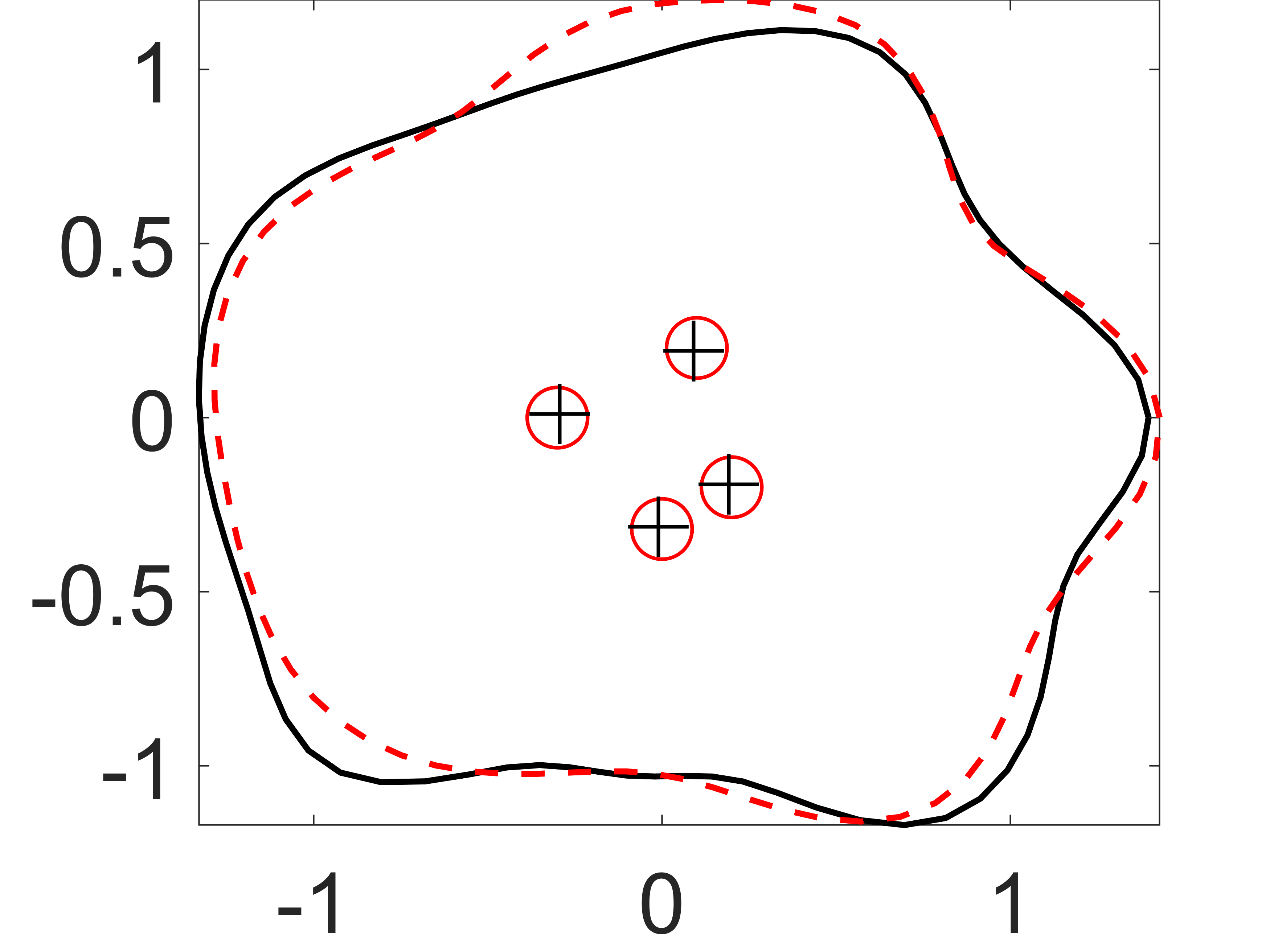}}
	    \subfigure[]{\includegraphics[width=0.3\textwidth]{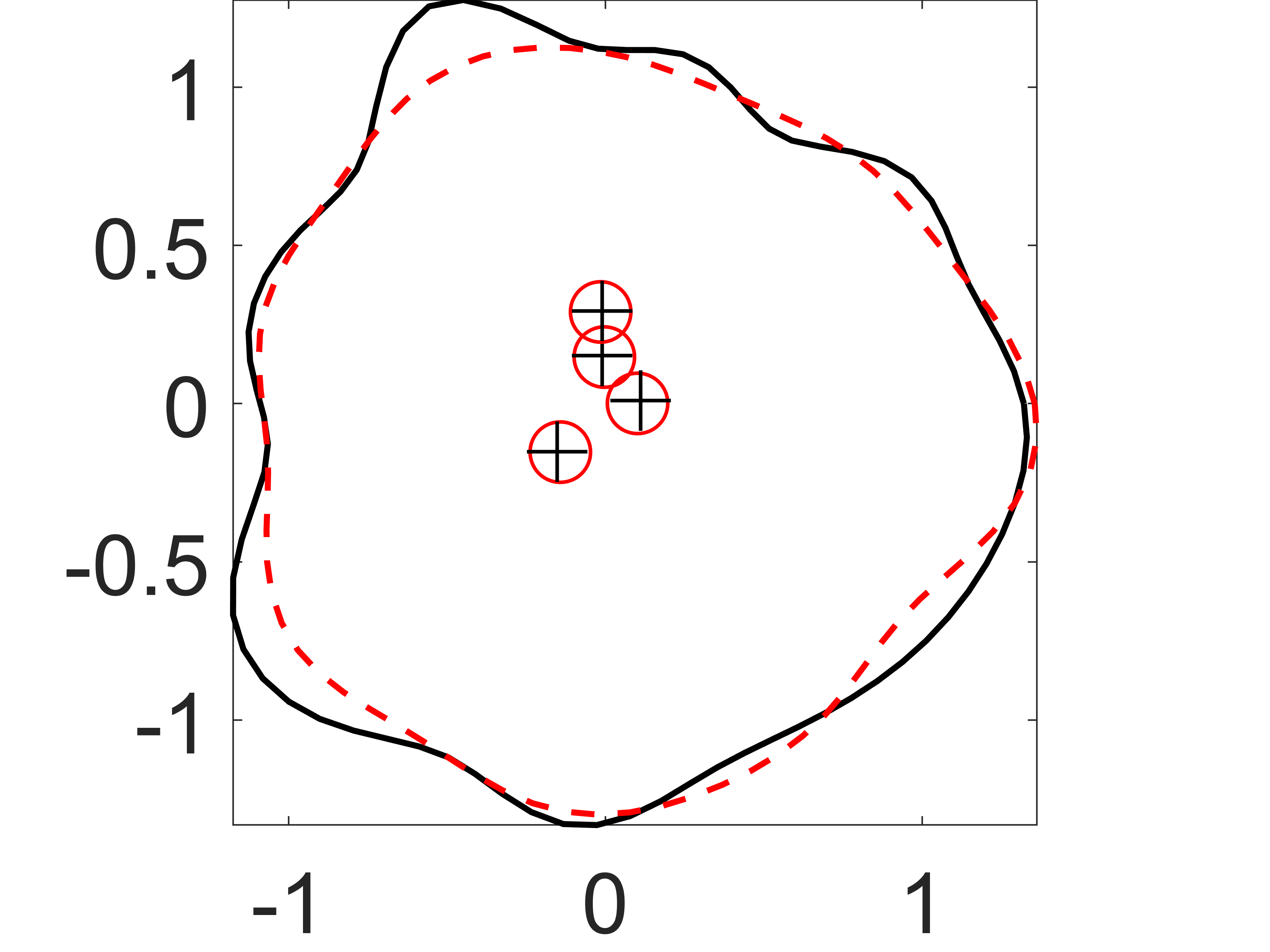}}
	    \subfigure[]{\includegraphics[width=0.3\textwidth]{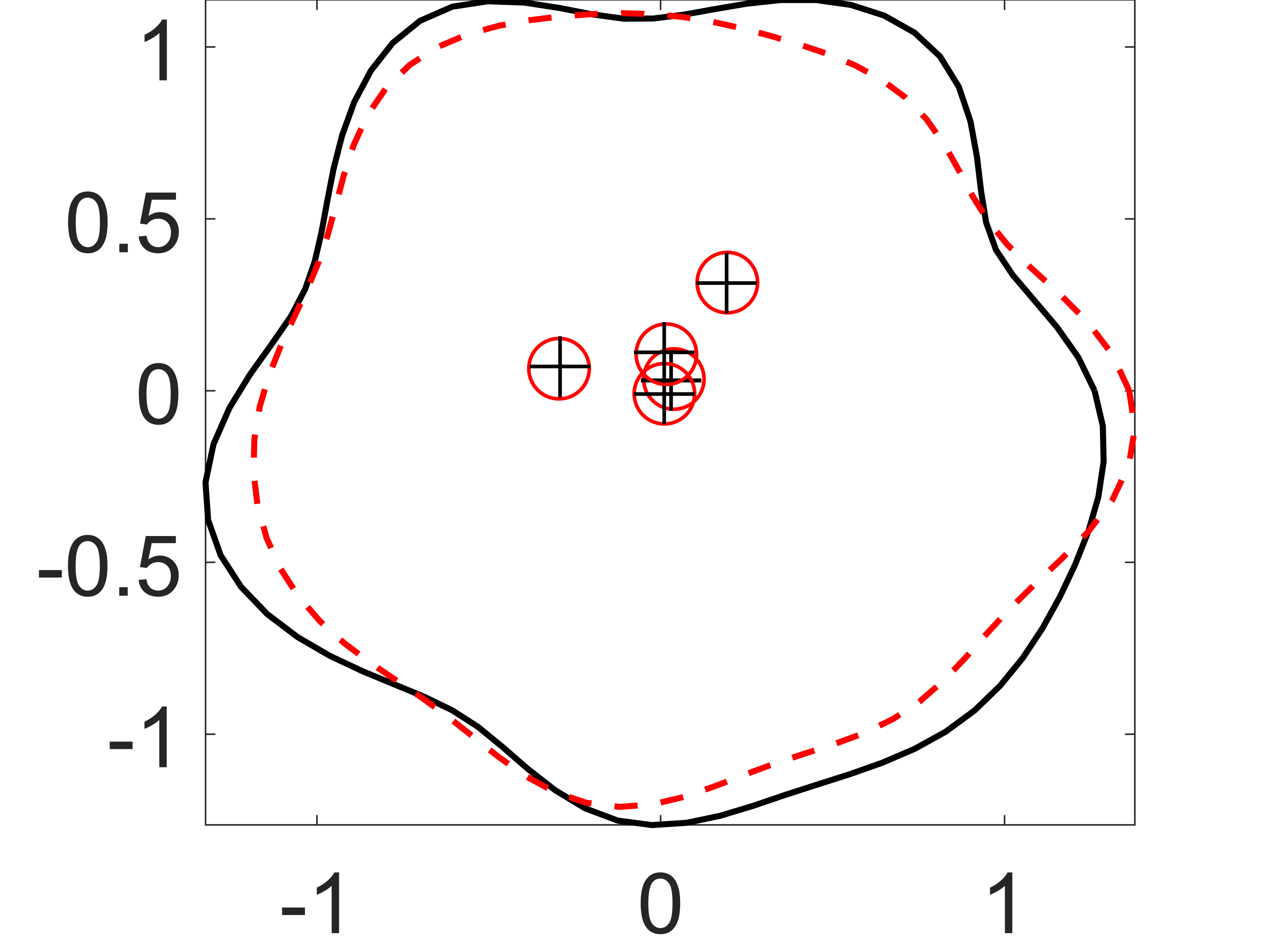}}
	    \caption{Reconstruction of the non-symmetric cavities together with 4 or 5 point sources. Row 1: configurations; Row 2: reconstructions.}
	\label{fig:random2}
    \end{figure}

    \begin{table}
    	\centering
    	\caption{The reconstruction errors for random cavities.}\label{tab:error_random}
    	\begin{tabular}{c|cccccc}
    		\toprule
    		Model &  \cref{fig:random}(a)   &  \cref{fig:random}(b)   &   \cref{fig:random}(c)   &  \cref{fig:random2}(a)   &  \cref{fig:random2}(b)   &   \cref{fig:random2}(c) \\ \midrule
    		Error & $5.84\%$ & $4.24\%$ & $5.76\%$ & $6.41\%$ & $5.70\%$ & $6.84\%$ \\ \bottomrule
    	\end{tabular}
    \end{table}

	\begin{table}
		\centering
		\caption{The source locations in \cref{fig:random2}(c).}\label{tab:location}
		\begin{tabular}{ccc}
			\toprule
			        &   Exact locations   & Reconstructed locations \\ \midrule
			Point 1 & $(-0.2959, 0.0640)$ &   $(-0.2929, 0.0707)$   \\
			Point 2 & $(0.0381, 0.0338)$  &   $(0.0303, 0.0303)$    \\
			Point 3 & $( 0.1937, 0.3146)$ &   $(0.1919, 0.3131)$    \\
			Point 4 & $(0.0106, -0.0089)$ &   $(0.0101, -0.0101)$   \\
			Point 5 & $(0.0162, 0.1060)$  &   $(0.0101, 0.1111)$    \\ \bottomrule
		\end{tabular}
	\end{table}

    \section{Conclusions}\label{sec:conclusion}

    We propose a numerical method to tackle the acoustical co-inversion problem of imaging a scattering cavity as well as its internal point sources. A key step to resolve the intractable reconstruction is the decoupling of source and scattering components in the intertwined co-inversion system. To this end, we deploy a new model configuration consisting of two measurement curves. The twinned curves play a significantly important role in compensating for the severe lack of information and therefore the resulting system of layer potentials is capable of untangling the interlocked unity of source and cavity. Then the decoupled subproblems can be individually solved by the modified optimization and sampling schemes. It is worthwhile to point out that the overall flow of our algorithm does not rely on any solution of the forward problem or alternating iteration between the source and cavity. These advantages greatly facilitate the implementation of the inversion process. Theoretical foundations such as uniqueness and stability issues of the method are established. Finally, the promising features of our approach are validated by extensive numerical examples.

    In our opinion, the proposed divide-and-conquer framework paves the way for other more complicated scenarios of co-inversion problems, for instance, co-inversion problems with phaseless data, electromagnetic/elastic models and the time-dependent co-inversion problems. We hope to report more novel findings in these intriguing directions in the future.

	\section*{Acknowledgments}
	
	D. Zhang and Y. Wang were supported by NSFC grant 12171200. Y. Guo and Y. Chang were supported by NSFC grant 11971133.
	

\end{document}